\documentclass[11pt,numbers]{article}

\usepackage{stmaryrd}
\SetSymbolFont{stmry}{bold}{U}{stmry}{m}{n}

\usepackage[margin=1in]{geometry}

\usepackage[utf8]{inputenc}
\usepackage[T1]{fontenc}
\usepackage[american]{babel}
\usepackage{microtype}
\usepackage{lmodern}

\usepackage[table,xcdraw]{xcolor}  

\usepackage{amsmath, amssymb, amsthm, mathtools, mathrsfs, bbm, bm}
\usepackage{nicefrac}
\usepackage{dsfont}
\usepackage{stmaryrd} 

\usepackage{graphicx}
\usepackage{subcaption}
\usepackage{float}
\usepackage{svg}
\usepackage{tikz}
\usepackage{pgfplots}
\pgfplotsset{compat=newest}

\usepackage{booktabs}
\usepackage{multirow}
\usepackage{threeparttable}

\usepackage{algorithm}
\usepackage{algpseudocode}

\usepackage{mdframed}
\newtheorem{theorem}{Theorem}
\newtheorem{definition}{Definition}
\newtheorem{assumption}{Assumption}
\newtheorem{lemma}{Lemma}
\newtheorem{remark}{Remark}

\usepackage{natbib}
\renewcommand{\bibsection}{\subsubsection*{References}}
\let\cite\citep

\usepackage{url}
\usepackage{hyperref}
\usepackage{cleveref}

\usepackage{enumitem}

\usepackage{xspace}
\usepackage{pifont}

\usepackage{authblk}

\begin{document}
\title{\bfseries Generalization in Representation Models via Random Matrix Theory: Application to Recurrent Networks}

\author[1,2]{Yessin Moakher}
\author[1]{Malik Tiomoko}
\author[3]{Cosme Louart}
\author[4]{Zhenyu Liao}

\affil[1]{Huawei Noah’s Ark Lab, Huawei Technologies, Paris, France}
\affil[2]{École Polytechnique, France}
\affil[3]{Chinese University of Hong Kong, Shenzhen, China}
\affil[4]{Huazhong University of Science and Technology, Wuhan, China}
\date{}
\maketitle






\begin{abstract}
We first study the generalization error of models that use a fixed feature representation (frozen intermediate layers) followed by a trainable readout layer. This setting encompasses a range of architectures, from deep random-feature models to echo-state networks (ESNs) with recurrent dynamics. Working in the high-dimensional regime, we apply Random Matrix Theory to derive a closed-form expression for the asymptotic generalization error.
We then apply this analysis to recurrent representations and obtain concise formula that characterize their performance.
\emph{Surprisingly}, we show that a linear ESN is equivalent to ridge regression with an exponentially time-weighted (“memory”) input covariance, revealing a clear inductive bias toward recent inputs. Experiments match predictions: ESNs win in low-sample, short-memory regimes, while ridge prevails with more data or long-range dependencies. Our methodology provides a general framework for analyzing overparameterized models and offers insights into the behavior of deep learning networks.

\end{abstract}
\section{Introduction}

Deep learning has achieved remarkable empirical success across a wide range of applications. Despite their extreme overparameterization, modern neural networks often generalize well, a phenomenon that challenges classical statistical learning theory. In particular, recent observations of double descent behavior reveal that conventional machine learning arguments fail to capture the complexities of high-dimensional learning dynamics. This has motivated the development of theoretical frameworks aimed at understanding when and why overparameterized models generalize.

A promising avenue in this direction is provided by random features models, which were originally introduced as scalable approximations to kernel methods \cite{NIPS2007_013a006f}. Over time, these models have also been studied as surrogates for neural networks \cite{mei2020generalizationerrorrandomfeatures}, offering a simplified yet powerful framework for understanding generalization. In these models, input vectors are mapped through a random nonlinear transformation, after which only a linear readout layer is trained. This decoupling of representation and readout retains much of the expressive power of neural networks while allowing for deeper theoretical analysis.

Building on this line of work, we focus on the broader question of learning under fixed representations. In this setting, the representation function is fixed a priori, and learning occurs exclusively through the readout. This setting arises in various contexts, including random-feature models \cite{NIPS2007_013a006f},  
random intermediate layers in deep networks \cite{schroder2023deterministic},  
and reservoir computing architectures such as \emph{Echo State Networks} (ESNs) \cite{article}.  
Its popularity stems from the practical advantages it offers for analysis.

However, despite its prevalence, the generalization properties of such models particularly when the feature map is structured rather than i.i.d. are still not fully understood. In the reccurent setting, \cite{pmlr-v48-couillet16} analyzed ESNs trained and later tested on a single time series, assuming both the input and target to be independent and deterministic, and derived closed-form expressions for the asymptotic mean-square error. While this provided valuable insights, the restrictive assumptions of independence and determinism limit its applicability.

In recent years, Random Matrix Theory (RMT) has emerged as a central tool for analyzing learning algorithms in the high-dimensional regime, where the number of samples and the feature dimension grow proportionally. RMT provides precise asymptotic predictions for performance metrics such as training and test error \cite{couillet2022random} and has been also applied to estimation problems, including for multi-task regression \cite{IlbertEtAl2024}.
\paragraph{Our approach.}
We address these gaps by developing a unified RMT framework for models with \emph{arbitrary fixed feature representations}, 
including structured, recurrent, linear and non linear maps, under the assumption that the transformed features form a \emph{concentrated random vector} \cite{Ledoux2001} and does not require i.i.d.\ or full-rank projections. Unlike most prior work on static models, we tackle the 
\emph{recurrent} setting, where the representation arises from a neural 
network with temporal dynamics.  

\paragraph{Contributions.}
Our main contributions are:
\begin{enumerate}
    \item \textbf{General risk characterization (\Cref{thm:asymp_expressions}).} 
   We derive a closed-form asymptotic limit for the out-of-sample risk of ridge regression on any fixed representation under the concentrated random vector assumption, including second-order deterministic equivalents obtained via concentration-of-measure arguments.
    \item \textbf{Specialization to ESNs (\Cref{thm:linear_esn}).} 
    We derive a closed-form performance characterization for Linear Echo State Recurrent Networks.
    \item \textbf{Insights.} We show that the test risk of  a linear ESN is equivalent to ridge regression on an \emph{exponentially time-weighted} (“memory”) input covariance and derive intuition on why we don't observe double descent with Linear ESN.
\end{enumerate}

\section{Related Work}

Our study is connected to the expanding literature that uses 
\emph{random matrix theory} (RMT) to analyze machine learning models 
in high-dimensional settings. Variations across studies mainly stem from 
differences in the assumed distribution of the features $X$ and from the 
relationship between $X$ and the target $Y$ (often referred to as the 
\emph{teacher--student} framework).

Early analyses often assumed that $Y$ was \emph{deterministic}. 
For instance, \cite{pmlr-v48-couillet16} examined the dynamics of a 
zero-shot linear echo state network (ESN) under this setting. 
Similarly, \cite{louart2018random} investigated a Gram random matrix 
student model with predictions of the form 
$\hat{Y} = \sigma(WX)$ to study random neural networks, 
also assuming deterministic targets $Y$.  
In the same spirit, \cite{NIPS2017_0f3d014e} studied the Gram matrix 
$\sigma(WX)$ when both the data $X$ and targets $Y$ were independent 
Gaussian variables.

Subsequent work shifted towards settings where $Y$ is 
\emph{linearly dependent} on $X$, such as
\[
    Y = \theta^\star X + \epsilon.
\]
In the case of linear student predictors 
$\hat{Y} = \hat{\theta} X$ corresponding to ridge regression models 
a variety of theoretical analyses have been carried out, 
highlighting phenomena such as the \emph{double descent} 
of the test error curve.  
Early results often assumed \emph{isotropic} feature distributions. 
For example, \cite{10.3150/14-BEJ609} studied ridge regression 
when the inputs $\boldsymbol{x_i}$ were sampled from an isotropic Gaussian 
distribution $\boldsymbol{x_i} \sim \mathcal{N}(0, I_d)$.  
Later works, such as \cite{10.5555/3495724.3496572} and 
\cite{pmlr-v130-richards21b}, extended these results to more general 
covariance structures. \cite{10.1214/21-AOS2133} derived insights on the ridgeless least squares interpolation.

Another closely related line of work considers \emph{random projections} prior to learning.  
For instance, \cite{doi:10.1137/23M1558781} analyzed models where the 
training data are transformed via random matrices, considering settings 
of the form $WSX$ where $S$ is a random projection matrix.  
Such formulations naturally connect to the case of linear ESNs, 
in which the reservoir acts as a fixed, structured projection of the input.  
However, most existing analyses assume i.i.d. entries in the projection 
matrix and often full-rank transformations.

\cite{mei2020generalizationerrorrandomfeatures} adds a non linear component the teacher model and perform ridge regression on random features. 
More recently, \cite{NEURIPS2023_38a1671a} study the problem of learning a polynomial target function when data is provided with a spiked covariance structure $\boldsymbol{x_i} \sim \mathcal{N}(0, I_d+\boldsymbol{\theta \mu\mu^\top} )$.
\paragraph{Notation} Throughout the paper, we use capital letters  to denote matrices, lowercase letters  for scalars, and lowercase bold letters for vectors. 
The Frobenius norm is denoted by \(\|\cdot\|_F\), the operator norm by \(\|\cdot\|\), and the Euclidean norm by \(\|\cdot\|_2\).
We say that $u = O(v)$ if the ratio $u/v$ remains bounded in this limit. 

\paragraph{Organization} The remainder of the paper is organized as follows. 
Section~\ref{sec:problem_setting} introduces the problem setting and defines the teacher and student models. 
Section~\ref{sec:main_technical_results} states our main assumptions and theoretical risk characterization, including comparisons between models. 
Section~\ref{sec:experiments} presents experiments supporting our theory, and Section~\ref{sec:conclusion} concludes.

\section{Problem Setting}
\label{sec:problem_setting}

We consider a supervised learning task in which the goal is to learn a predictor 
that maps an input sequence to an output vector, based on a finite set of training samples.

\paragraph{Training and test setup.}
Let $\{ (\mathbf{u}_i, \mathbf{y}_i) \}_{i=1}^N$ be $N$ independent and identically 
distributed (i.i.d.) training pairs, where
\[
    \mathbf{u}_i \in \mathbb{R}^T, 
    \quad \mathbf{y}_i \in \mathbb{R}^q.
\]

Throughout, both the training and test pairs are assumed to be generated from the following model.

\begin{definition}[Noisy linear model]
\label{def:noisy_linear_model}
An input–output pair $(\mathbf{u}, \mathbf{y}) \in \mathbb{R}^T \times \mathbb{R}^q$ 
is said to follow a \emph{noisy linear model} if
\begin{equation}
    \mathbf{y} = \Theta_\ast^{\top} \mathbf{u} + \boldsymbol{\epsilon},
    \label{eq:model}
\end{equation}
where:
\begin{itemize}
    \item $\Theta_\ast \in \mathbb{R}^{T \times q}$ is the (unknown) ground-truth parameter matrix,
    \item $\boldsymbol{\epsilon} \in \mathbb{R}^q$ is a noise vector with i.i.d.\ entries of zero mean and variance $\sigma^2$, and $\boldsymbol{\epsilon}$ is independent of $\mathbf{u}$.
\end{itemize}
This model can be viewed as a linearization of more general nonlinear models in high dimensions 
(see, e.g., \cite{NIPS2017_0f3d014e}, \cite{NEURIPS2023_abccb8a9}).
\end{definition}

\paragraph{Feature representation.}
Rather than using the raw input $\mathbf{u}$ directly, we first transform it via a fixed (possibly nonlinear) 
representation map:
\begin{equation}
    F : \mathbb{R}^{T} \to \mathbb{R}^{n}, 
    \qquad \mathbf{z} = F(\mathbf{u}).
    \label{eq:feature_representation}
\end{equation}
Examples of such $F$ include:
\begin{itemize}
    \item the reservoir state of a (linear or nonlinear) Echo State Network (ESN),
    \item random feature maps,
    \item intermediate layers of a pretrained network.
\end{itemize}
In this work, $F$ is \emph{fixed and known}; only the final linear readout is learned from data.

\paragraph{Linear readout with ridge regularization.}
Let $Z = [\mathbf{z}_1, \dots, \mathbf{z}_N] \in \mathbb{R}^{n \times N}$ 
be the matrix of feature vectors and 
$Y = [\mathbf{y}_1, \dots, \mathbf{y}_N] \in \mathbb{R}^{q \times N}$ 
the matrix of corresponding targets.  
We estimate the output weights via ridge-regularized regression:
\begin{equation}
    \hat{W}_{\!\text{out}}
    := \arg\min_{W \in \mathbb{R}^{q \times n}}
        \frac{1}{N} \sum_{i=1}^N \lVert \mathbf{y}_i - W \mathbf{z}_i \rVert_2^2
        + \lambda \lVert W \rVert_F^2
    = \frac{1}{N} Y Z^{\top} \left( \frac{1}{N} Z Z^{\top} + \lambda I_n \right)^{-1},
    \label{eq:ridge_readout}
\end{equation}
where $\lambda > 0$ is the regularization parameter.

\paragraph{Prediction.}
Given a new test input $\mathbf{u}' \in \mathbb{R}^T$, we compute its feature vector 
$\mathbf{z}' := F(\mathbf{u}')$ and output
\begin{equation}
    \hat{\mathbf{y}}' = \hat{W}_{\text{out}} \, \mathbf{z}'.
    \label{eq:def_z0}
\end{equation}

\begin{definition}[Out-of-sample risk]
\label{def:out-of-sample}
The \emph{out-of-sample risk} of the predictor is the mean squared prediction error on an independent test sample $(\mathbf{u'},\mathbf{y'})$:
\begin{equation}
    \mathcal{R} := \frac{1}{q} \, \mathbb{E}\left[ \left\| \mathbf{y}' - \hat{\mathbf{y}}' \right\|_2^2 \right],
    \label{eq:risk_def}
\end{equation}
where the expectation is taken over both training and test data.
\end{definition}

\section{Main Technical Results}
\label{sec:main_technical_results}
\subsection{Asymptotic characterization of out-of-sample risk}

We define the following quantities: $\Sigma_u := \mathbb{E}[\mathbf{uu}^\top] = \operatorname{Cov}(\mathbf{u})+\mathbb{E}(\mathbf{u})\mathbb{E}(\mathbf{u})^\top \in \mathbb{R}^{T \times T}$, $\Sigma_{z}:=\mathbb{E}[\mathbf{zz}^\top] \in \mathbb{R}^{n \times n}$, $\Sigma_{uz} := \mathbb{E}[\mathbf{uz}^\top] \in \mathbb{R}^{T \times n}$ and the resolvent $Q:= (\frac{Z Z^\top}{N}+\lambda I_n)^{-1}$.

\paragraph{Assumptions.}In order to use Random Matrix Theory (RMT) tools, we make assumptions on the data distribution
and the asymptotic regime.

\begin{definition}[Concentrated random vector]\label{def:concentrated}
A random vector $\mathbf{x} \in \mathbb{R}^n$ is said to be \emph{concentrated} if there exist constants $C, c > 0$, independent of $n$, such that: for every $1$-Lipschitz function $f : \mathbb{R}^n \to \mathbb{R}$ and all $t \geq 0$,
\[
\mathbb{P}\bigl( \lvert f(\mathbf{x}) - \mathbb{E}[f(\mathbf{x})] \rvert \geq t \bigr) \leq C e^{-c t^2}.
\]
\end{definition}

\begin{assumption}\label{assum:concentration}
The representation vector $\mathbf{z} \in \mathbb{R}^T$ is a concentrated random vector in the sense of the definition \ref{def:concentrated}  and satisfies $\|\mathbb{E}[\mathbf{z}]\|_2 = O(1)$.
\end{assumption}
This class includes Gaussian vectors with covariance matrices bounded in operator norm, uniform vectors on the sphere, and any Lipschitz transformation thereof (e.g., features from GANs \cite{seddik2020random}). This assumption allows us to apply the Hanson--Wright inequality in the context of random matrix theory with non-isotropic vectors \cite{Adamczak2014ANO}. Other works, (e.g. \cite{doi:10.1137/23M1558781}) make the assumption $z = \Sigma^{1/2} S$, where $S$ is i.i.d.\ subgaussian and $\Sigma$ bounded, in order to apply the Hanson--Wright inequality in the isotropic case \cite{rudelson2013hansonwrightinequalitysubgaussianconcentration}. The two setups overlap (e.g., Gaussians), but neither contains the other in full generality, and in most cases the same results could be derived under either assumption.

\begin{assumption}\label{assum:asymptotic}
We work in the classical random matrix theory proportional asymptotics regime, where the number of reservoirs $n$ and the number of samples $N$ diverge proportionally. That is,  
$
\frac{n}{N} \xrightarrow[]{} \gamma \in (0, \infty), \quad \text{as } n, N \to \infty.$

\end{assumption}
\paragraph{Asymptotic Expressions.}
We will make use of the following quantities:
\begin{align*}
    \bar Q &:= \left(\frac{\Sigma_z}{1+\delta} + \lambda I_n\right)^{-1}, \quad 
    \delta := \frac{1}{N} \operatorname{Tr}(\Sigma_z \bar Q), \quad \alpha := \frac{1}{N} \operatorname{Tr} \left( \frac{\Sigma_z}{1 + \delta} \bar Q \frac{\Sigma_z}{1 + \delta} \bar Q \right)
\end{align*}

\begin{theorem}[Fixed Representation Generalization]\label{thm:asymp_expressions}
Let \(\mathbf{u} \in \mathbb{R}^T\) be an input vector, and let \(\mathbf{z} = F(\mathbf{u}) \in \mathbb{R}^n\) be a representation vector obtained through a transformation of \(\mathbf{u}\).

Under the linear model of Definition~\ref{def:noisy_linear_model}, where predictions are defined as in Equation~\ref{eq:def_z0}, and under Assumptions~\ref{assum:concentration} and~\ref{assum:asymptotic}, the following expressions hold:
$
    \mathcal{R} = \mathcal{B}^2 + \mathcal{V}+\sigma^2$ where
\begin{align*}
&\mathcal{B}^2  \to \frac{1}{1 - \alpha} \Bigg(
    \operatorname{Tr}(\Theta_\ast^{\top} \Sigma_u \Theta_\ast) 
    - \frac{2}{1 + \delta} \operatorname{Tr} \left( \Theta_\ast^{\top} \Sigma_{uz} \bar Q \Sigma_{uz}^\top \Theta_\ast \right) 
     + \frac{1}{(1 + \delta)^2} \operatorname{Tr} \left( \Theta_\ast^{\top} \Sigma_{uz} \bar Q \Sigma_z \bar Q \Sigma_{uz}^\top \Theta_\ast \right)
\Bigg)  \\
&\mathcal{V} \to \sigma^2  \frac{\alpha}{1 - \alpha} 
\end{align*}
\end{theorem}
\begin{proof}
See Appendix \ref{subsec:proof_of_lem_bias-variance_decomp}. 
The argument combines Sherman--Morrison identities to disentangle the resolvent $Q$ from the other random variables, concentration of quadratic forms to replace terms by their expectations, 
and deterministic equivalents for the limit. 
\end{proof}

\begin{remark}[Special Cases]\label{rm:thm1}
    Under the assumption that $\mathbf{u}$ is concentrated, note that for $F$ equal to the identity we recover ridge regression. For $F$ given by an ESN or even a feedforward neural network, under the assumptions that the weight matrices are normalized, and activation functions $f$ are Lipschitz with $f(0) = 0$, we have that $\mathbf{z}$ is concentrated with $\|\mathbb{E}[\mathbf{z}]\|_2 = O(1)$. The same holds for random projections with bounded operator norm, not necessarily i.i.d. sub-Gaussian.
\end{remark}
\subsection{Application to Recurrent Models Representations}

The goal of this section is to apply Theorem~\ref{thm:asymp_expressions} to the case of recurrently generated representations. 
To this end, we consider an Echo State Network (ESN), a class of recurrent neural networks (RNNs) designed for sequential data processing, particularly in time series forecasting, speech recognition, and dynamical system modeling.
\begin{definition}[Echo State Network]
An Echo State Network consists of:
\begin{enumerate}
    \item A fixed input layer \( \boldsymbol{w_{ in}} \in \mathbb{R}^{n} \) that maps the input $u_i(t) \in \mathbb{R}$ into an \( n \)-dimensional space.
    \item A fixed recurrent reservoir layer \( W \in \mathbb{R}^{n \times n} \) that captures the temporal dependency of data.
    \item A trainable output layer \( W_{\rm out} \in \mathbb{R}^{q \times n} \) that maps reservoir states to predictions.
\end{enumerate}
For an input $u_i(t)$, the reservoir state of the ESN is denoted $\mathbf{x_i}(t) \in \mathbb{R}^{n}$ and evolves according to:
\begin{equation}\label{eq:reservoir_dynamics}
    \mathbf{x_i(t)} = f\left( u_i(t) \boldsymbol{w_{ in}} + W \mathbf{x_i}(t-1)\right),
\end{equation}
where \(\mathbf{ x_i}(0) = 0 \), and \( f(\cdot) \) is the activation function such as $f(t) = \tanh(t)$ or $f(t) = {\rm ReLU}(t) = \max(t,0)$ that applies entry-wise.
The representation vector \( \mathbf{z_i} \in \mathbb{R}^n \) is defined as the reservoir state at the (final) time step $T$: \( \mathbf{z_i} = \mathbf{ x_i}(T) \). 
\end{definition}
Compared to standard RNNs, ESNs fix the input and recurrent weights (typically drawn randomly) and train only the output layer.

\paragraph{Linear ESN}
For linear ESN (i.e $f$ is the identity), we have a closed form for the relationship between $U$ and $Z$, that is :
\begin{equation}\label{eq:S}
    Z = SU, \quad \text{where } S = \begin{bmatrix}
W^{T-1} \boldsymbol{w_{\rm in}} ,
W^{T-2} \boldsymbol{w_{\rm in}} ,
\ldots ,
W^0 \boldsymbol{w_{\rm in}}
\end{bmatrix}  \in \mathbb{R}^{n \times T}
\end{equation}
The matrix $U$ is being projected by a sort of a Kalman controllability matrix $S$. In that case we have a closed form expression of $\Sigma_{uz} =  \Sigma S^\top $ and $\Sigma_z = S \Sigma S^\top$.
\paragraph{Asymptotic Expressions.}
To ensure dynamical stability and guarantee that $\boldsymbol{z}$ remains a concentrated vector under the assumption that $\boldsymbol{u}$ is concentrated, we impose the following assumption.
\begin{assumption}[Linear ESN setting]
\label{assump:linear_esn}
We consider a linear ESN with recurrent matrix
$W = \frac{W_0}{\varphi \, \rho(W_0)},$ where $W_0 \in \mathbb{R}^{n \times n}$ has i.i.d.\ entries $W_{0,ij} \sim \mathcal{N}(0,1)$, $\rho(W_0)$ is its spectral radius, and $\varphi > 0$ controls the effective spectral radius.  
The input weight vector $\boldsymbol{w}_{\mathrm{in}}$ has i.i.d.\ entries normalized $\mathcal{N}(0,\frac{1}{n})$.
\end{assumption}
The parameter $\varphi < 1$ plays a key role in the dynamical stability of the ESN. Larger $\varphi$ can capture longer temporal dependencies but may also lead to instability if the system enters a regime of diverging activations.
\begin{theorem}[Linear ESN Generalization]
\label{thm:linear_esn}
Let $(\mu_i, \boldsymbol{v_i})$ be the eigenvalue–eigenvector pairs of
$
\Sigma_u^{1/2} \, \operatorname{diag}(\varphi^{i-T})_{1 \leq i \leq T} \, \Sigma_u^{1/2}.$
Under the assumption that $\boldsymbol{u} \in \mathbb{R}^T$ is concentrated, and under Assumptions~\ref{assum:asymptotic} and~\ref{assump:linear_esn}, the out-of-sample risk decomposes as $
    \mathcal{R} = \mathcal{B}^2 + \mathcal{V}+\sigma^2$, where
\begin{equation*}
\mathbb{E}_{W,\boldsymbol{w}_{\mathrm{in}}}[\mathcal{B}^2] \to \frac{1}{1-\alpha}
\sum_{i=1}^{T} \frac{\kappa^2}{(\mu_i+\kappa)^2} \, \lVert \Theta_\ast^{\top} \Sigma_u^{1/2} v_i \rVert_2^2, 
\qquad
\mathbb{E}_{W,\mathbf{w}_{\mathrm{in}}}[\mathcal{V}] \to
\sigma^2  \frac{\alpha}{1-\alpha}, 
\end{equation*}
with
$\kappa := \lambda (1+\delta)$ and $
\alpha := \sum_{i=1}^T\frac{\mu_i}{N(\mu_i+\kappa)^2}.$
\end{theorem}
\begin{proof}
See Appendix \ref{proof:th2}.  
We first show that $S^\top S$ is concentrated, and then compute the limit $M_\infty$ of its expectation using random matrix theory. Next, we prove that $\mathcal{R}$, viewed as a function of $S^\top S$, is Lipschitz with constant of order $O(1)$. This allows us to replace, in the limit, $\mathbb{E}[\mathcal{R}(S^\top S)]$ with $\mathcal{R}(M_\infty)$.
\end{proof}

\begin{remark}[Difference with Ridge regression and Interpretation]\label{remark:ESN_interpretation} In the case of ridge regression ($\Sigma_u = \Sigma_{uz} = \Sigma_z$), we obtain the same expression but the same expression with $(\mu_i, \boldsymbol{v_i})$ the eigenvalue–eigenvector pairs of
$\Sigma_u$ instead of $
\Sigma_u^{1/2} \, \operatorname{diag}(\varphi^{i-T})_{1 \leq i \leq T} \, \Sigma_u^{1/2}.$ Which proves that Linear ESN is equivalent to ridge regression but with exponentially time-weighted version covariance, which down-weights older inputs. This induces a short-memory bias, distorting the input statistics and discarding part of the long-term information.
\end{remark}
\begin{remark}[Optimal Regularization]
The asymptotic risk depends on the scalar $\kappa = \lambda (1+\delta)$, which can be optimized as a real-valued function independently of the fixed point $\delta$ (which itself depends on $\lambda$). Once the optimal $\kappa$ is found, the corresponding $\delta$ can be recovered from
$\frac{\delta}{1+\delta} = \frac{1}{N} \sum_{i=1}^T \frac{\mu_i}{\mu_i + \kappa}.$
This then yields the optimal $\lambda$.  

In general, there is no closed-form expression for $\lambda^\star$ for arbitrary $\Sigma_u$. However, when $\Sigma_u = I_T$, we obtain
$
\lambda^\star = \frac{T}{N} \cdot \mathrm{SNR}, 
\quad \mathrm{SNR} = \frac{\sum_{i=1}^T \|\Theta^{\ast\top} \mathbf{v_i}\|^2}{\sigma^2}.
$
\end{remark}

\section{Experiments: Comparing Ridge Regression and Linear ESN}
\label{sec:experiments}
\subsection{Double Descent}
\label{sec:double_descent}
The \emph{double descent} phenomenon describes how the test error decreases, peaks near the interpolation threshold, and then decreases again as model complexity grows. This behavior is well established for linear predictors, both empirically and theoretically.
\paragraph{Key mechanism.}
The factor $(1 - \alpha)^{-1}$ in \Cref{thm:asymp_expressions} diverges as $\alpha$ tends to $1$.

Recall that
\[
\alpha \;=\; \frac{1}{N} \sum_{i=1}^T 
    \frac{\mu_i^2}{\left( \mu_i + \lambda(1+\delta) \right)^2}.
\]
Double descent is typically observed for $\lambda$ close to $0$ \cite{mei2020generalizationerrorrandomfeatures}. In this case,
\[
\alpha 
    \xrightarrow[\lambda \to 0]{} 
    \frac{1}{N} \sum_{i=1}^r 1
    \;=\; \frac{r}{N},
\]
where $r$ is the rank of
$
\Sigma_u^{1/2} \, \operatorname{diag}(\varphi^{i-T})_{1 \leq i \leq T} \, \Sigma_u^{1/2}.$
In classical ridge regression with $\varphi = 1$, the feature covariance matrix is full rank, and double descent occurs when $N = T$. In contrast, for a Linear ESN, the matrix $\Sigma_u^{1/2} \, \operatorname{diag}(\varphi^{i-T})_{1 \leq i \leq T} \, \Sigma_u^{1/2}$ is low rank, since $\varphi^{i-T}$ rapidly decays to $0$ in numerical computations, and thus double descent is not observed as show in \Cref{fig:double_descent}.

\begin{figure}[H]
    \centering
    \includegraphics[width=0.8\textwidth]{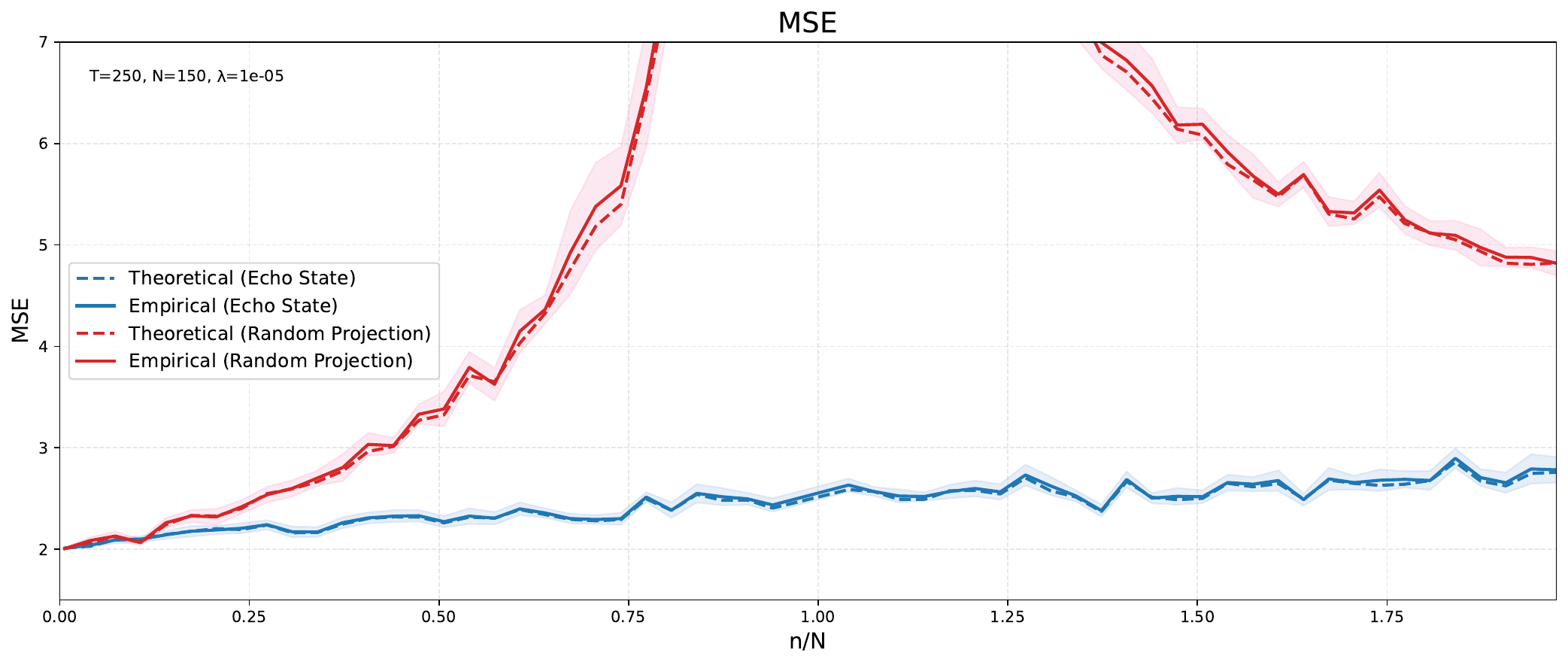}
    \caption{
Comparison of test error curves for Linear ESNs and random projections as a function of  $n/N$ . For random projections, the test error exhibits a clear double descent peak near the interpolation threshold ($n/N \approx 1$). In contrast, Linear ESNs show no such peak due to their effective low-rank feature covariance, as explained above.
    }
    \label{fig:double_descent}
\end{figure}

\subsection{Comparing ESNs and Ridge Regression}

The \Cref{remark:ESN_interpretation} indicates that Echo State Networks (ESNs) could surpass Ridge regression in scenarios characterized by \emph{limited data and short temporal dependencies}. 
This advantage arises from the inherent inductive bias of ESNs: they assume that the relevant information is encoded in the recent history of the input, effectively emphasizing short-term temporal correlations.

Ridge regression, when applied directly to raw input vectors or generic features, lacks this temporal structure. In low-data regimes, this absence of an inductive bias often leads to overfitting or suboptimal generalization. ESNs, in contrast, implicitly filter past inputs through their recurrent reservoir, efficiently capturing patterns from the recent past and making better use of scarce data.

\begin{figure}[H]
    \centering
    \includegraphics[width=0.7\textwidth,height=0.25\textheight]{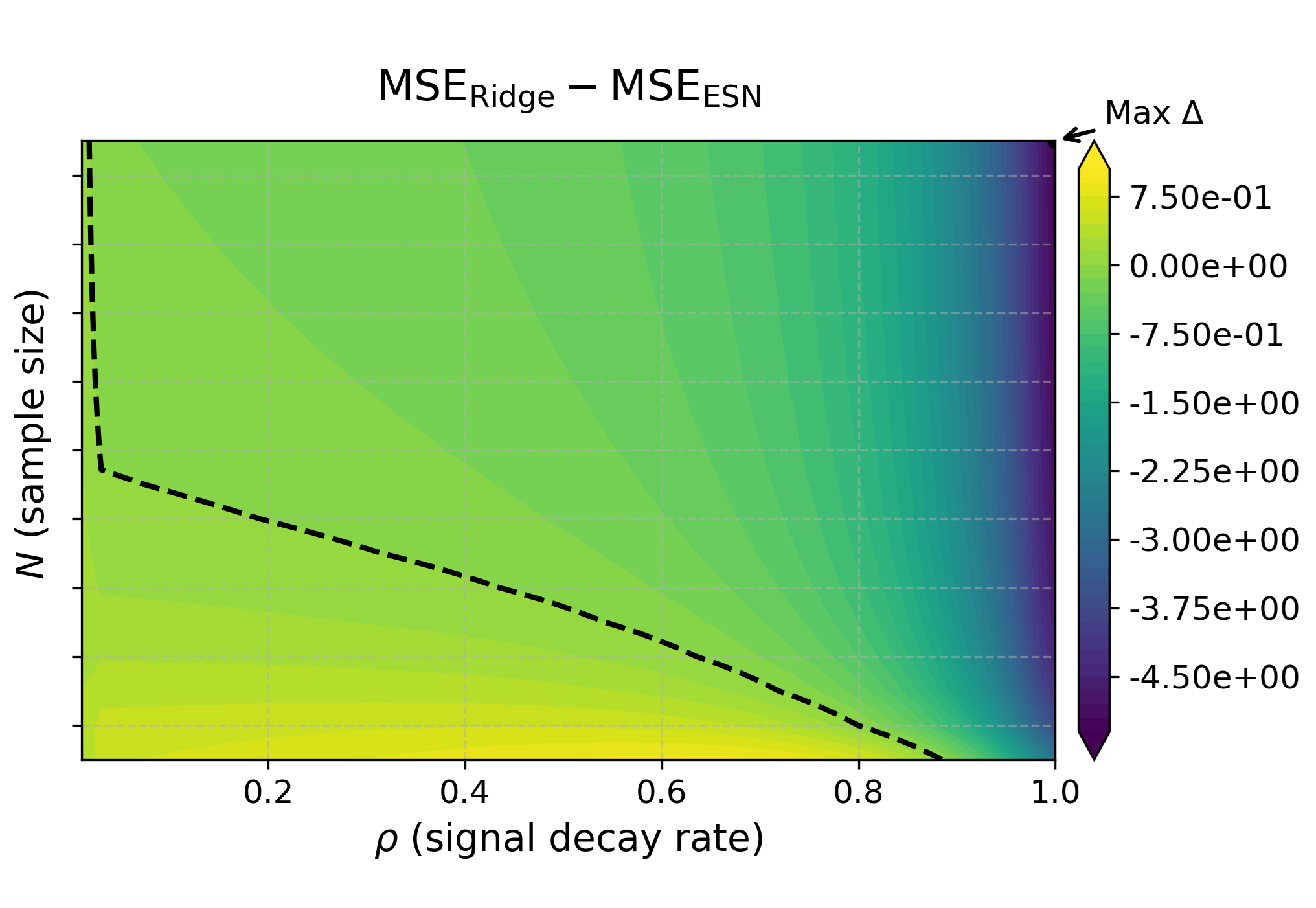}
\caption{
Comparison between ESNs and ridge regression.
The test error is shown as a function of the sample size $N$ (y-axis) and the signal decay rate $\rho$ (x-axis), where $\boldsymbol{\theta}^\star = (\rho^t,\, t \in \{1,\dots,T\})$ for both models. Shaded regions indicate the standard deviation over multiple runs. The dashed line marks the separation threshold. ESNs exhibit superior performance in the limited-data, short-memory regime.
}

    \label{fig:synthetic_comparison}
\end{figure}

In \Cref{fig:synthetic_comparison}, we confirm this intuition: ESNs consistently outperform Ridge regression when $T/N$ is small and when target function depends predominantly on short-term dependencies (small signal decay rate $\rho$).
In this regime, the exponential time-weighting of ESNs effectively acts as a regularizer, improving generalization by attenuating the influence of distant, noisy inputs.
However, as $N$ grows, Ridge regression eventually surpasses ESNs, since, as discussed in \Cref{remark:ESN_interpretation}, the ESNs inevitably discards part of the long-term information, which becomes increasingly useful when sufficient samples are available.
\section{Conclusion}
\label{sec:conclusion}
We introduced a general high-dimensional theory for ridge regression on arbitrary fixed representations, unifying and extending prior analyses of random features, pretrained models, and recurrent architectures. Applied to Echo State Networks, our framework yields simple closed-form predictions that match experiments, reveal their temporal inductive bias, and explain their absence of double descent. Our results show that Echo State Networks can outperform ridge regression in low-data, short-memory regimes due to their built-in temporal inductive bias, while ridge becomes advantageous as the sample size grows. 
\bibliography{neurips_2024.bib}

\begin{thebibliography}{27}
\providecommand{\natexlab}[1]{#1}
\providecommand{\url}[1]{\texttt{#1}}
\expandafter\ifx\csname urlstyle\endcsname\relax
  \providecommand{\doi}[1]{doi: #1}\else
  \providecommand{\doi}{doi: \begingroup \urlstyle{rm}\Url}\fi

\bibitem[Adamczak(2014)]{Adamczak2014ANO}
Radosław Adamczak.
\newblock A note on the hanson-wright inequality for random vectors with
  dependencies.
\newblock \emph{Electronic Communications in Probability}, 2014.

\bibitem[Alt et~al.(2021)Alt, Erdős, and Krüger]{Alt_2021}
Johannes Alt, László Erdős, and Torben Krüger.
\newblock Spectral radius of random matrices with independent entries.
\newblock \emph{Probability and Mathematical Physics}, 2021.

\bibitem[Ba et~al.(2023)Ba, Erdogdu, Suzuki, Wang, and
  Wu]{NEURIPS2023_38a1671a}
Jimmy Ba, Murat~A Erdogdu, Taiji Suzuki, Zhichao Wang, and Denny Wu.
\newblock Learning in the presence of low-dimensional structure: A spiked
  random matrix perspective.
\newblock In \emph{Advances in Neural Information Processing Systems}, 2023.

\bibitem[Bach(2024)]{doi:10.1137/23M1558781}
Francis Bach.
\newblock High-dimensional analysis of double descent for linear regression
  with random projections.
\newblock \emph{SIAM Journal on Mathematics of Data Science}, pages 26--50,
  2024.

\bibitem[Couillet and Liao(2022)]{couillet2022random}
Romain Couillet and Zhenyu Liao.
\newblock \emph{Random matrix methods for machine learning}.
\newblock Cambridge University Press, 2022.

\bibitem[Couillet et~al.(2016)Couillet, Wainrib, Ali, and
  Sevi]{pmlr-v48-couillet16}
Romain Couillet, Gilles Wainrib, Hafiz~Tiomoko Ali, and Harry Sevi.
\newblock A random matrix approach to echo-state neural networks.
\newblock In \emph{Proceedings of The 33rd International Conference on Machine
  Learning}, 2016.

\bibitem[Dandi et~al.(2023)Dandi, Stephan, Krzakala, Loureiro, and
  Zdeborov\'{a}]{NEURIPS2023_abccb8a9}
Yatin Dandi, Ludovic Stephan, Florent Krzakala, Bruno Loureiro, and Lenka
  Zdeborov\'{a}.
\newblock Universality laws for gaussian mixtures in generalized linear models.
\newblock In \emph{Advances in Neural Information Processing Systems}, 2023.

\bibitem[Dicker(2016)]{10.3150/14-BEJ609}
Lee~H. Dicker.
\newblock {Ridge regression and asymptotic minimax estimation over spheres of
  growing dimension}.
\newblock \emph{Bernoulli}, 2016.

\bibitem[Hastie et~al.(2022)Hastie, Montanari, Rosset, and
  Tibshirani]{10.1214/21-AOS2133}
Trevor Hastie, Andrea Montanari, Saharon Rosset, and Ryan~J. Tibshirani.
\newblock {Surprises in high-dimensional ridgeless least squares
  interpolation}.
\newblock \emph{The Annals of Statistics}, 2022.

\bibitem[Ilbert et~al.(2024)Ilbert, Tiomoko, Louart, Odonnat, Feofanov,
  Palpanas, and Redko]{IlbertEtAl2024}
Romain Ilbert, Malik Tiomoko, Cosme Louart, Ambroise Odonnat, Vasilii Feofanov,
  Themis Palpanas, and Ievgen Redko.
\newblock Analysing multi‐task regression via random matrix theory with
  application to time series forecasting.
\newblock In \emph{Advances in Neural Information Processing Systems
  (NeurIPS)}, 2024.

\bibitem[Jaeger(2001)]{article}
Herbert Jaeger.
\newblock The" echo state" approach to analysing and training recurrent neural
  networks-with an erratum note'.
\newblock \emph{Bonn, Germany: German National Research Center for Information
  Technology GMD Technical Report}, 2001.

\bibitem[Kemp et~al.(2009)Kemp, Mahlburg, Rattan, and
  Smyth]{kemp2009enumerationnoncrossingpairingsbit}
Todd Kemp, Karl Mahlburg, Amarpreet Rattan, and Clifford Smyth.
\newblock Enumeration of non-crossing pairings on bit strings, 2009.

\bibitem[Ledoux(2001)]{Ledoux2001}
Michel Ledoux.
\newblock \emph{The Concentration of Measure Phenomenon}.
\newblock American Mathematical Society, 2001.

\bibitem[Louart(2023)]{louart:tel-04116888}
Cosme Louart.
\newblock \emph{{Random matrix theory and concentration of the measure theory
  for the study of high dimension data processing.}}
\newblock PhD thesis, {Universit{\'e} Grenoble Alpes [2020-....]}, 2023.

\bibitem[Louart and Couillet(2021)]{louart2021concentrationmeasurelargerandom}
Cosme Louart and Romain Couillet.
\newblock Concentration of measure and large random matrices with an
  application to sample covariance matrices, 2021.

\bibitem[Louart and Couillet(2022)]{louart2022concentration}
Cosme Louart and Romain Couillet.
\newblock A concentration of measure and random matrix approach to
  large-dimensional robust statistics.
\newblock \emph{The Annals of Applied Probability}, 2022.

\bibitem[Louart et~al.(2018)Louart, Liao, and Couillet]{louart2018random}
Cosme Louart, Zhenyu Liao, and Romain Couillet.
\newblock A random matrix approach to neural networks.
\newblock \emph{The Annals of Applied Probability}, 28\penalty0 (2):\penalty0
  1190--1248, 2018.

\bibitem[Mei and Montanari(2020)]{mei2020generalizationerrorrandomfeatures}
Song Mei and Andrea Montanari.
\newblock The generalization error of random features regression: Precise
  asymptotics and double descent curve, 2020.

\bibitem[Pennington and Worah(2017)]{NIPS2017_0f3d014e}
Jeffrey Pennington and Pratik Worah.
\newblock Nonlinear random matrix theory for deep learning.
\newblock In \emph{Advances in Neural Information Processing Systems}, 2017.

\bibitem[Rahimi and Recht(2007)]{NIPS2007_013a006f}
Ali Rahimi and Benjamin Recht.
\newblock Random features for large-scale kernel machines.
\newblock In \emph{Advances in Neural Information Processing Systems}, 2007.

\bibitem[Richards et~al.(2021)Richards, Mourtada, and
  Rosasco]{pmlr-v130-richards21b}
Dominic Richards, Jaouad Mourtada, and Lorenzo Rosasco.
\newblock Asymptotics of ridge(less) regression under general source condition.
\newblock In \emph{Proceedings of The 24th International Conference on
  Artificial Intelligence and Statistics}, 2021.

\bibitem[Rudelson and
  Vershynin(2013)]{rudelson2013hansonwrightinequalitysubgaussianconcentration}
Mark Rudelson and Roman Vershynin.
\newblock Hanson-wright inequality and sub-gaussian concentration, 2013.

\bibitem[Schr{\"o}der et~al.(2023)Schr{\"o}der, Cui, Dmitriev, and
  Loureiro]{schroder2023deterministic}
Dominik Schr{\"o}der, Hugo Cui, Daniil Dmitriev, and Bruno Loureiro.
\newblock Deterministic equivalent and error universality of deep random
  features learning.
\newblock In \emph{International Conference on Machine Learning}. PMLR, 2023.

\bibitem[Seddik et~al.(2020)Seddik, Louart, Tamaazousti, and
  Couillet]{seddik2020random}
Mohamed El~Amine Seddik, Cosme Louart, Mohamed Tamaazousti, and Romain
  Couillet.
\newblock Random matrix theory proves that deep learning representations of
  gan-data behave as gaussian mixtures.
\newblock In \emph{International Conference on Machine Learning}, 2020.

\bibitem[Tao(2011)]{noauthororeditor}
Terence Tao.
\newblock \emph{Topics in random matrix theory}.
\newblock American Mathematical Society, 2011.

\bibitem[Vershynin(2018)]{Vershynin_2018}
Roman Vershynin.
\newblock \emph{High-Dimensional Probability: An Introduction with Applications
  in Data Science}.
\newblock Cambridge University Press, 2018.

\bibitem[Wu and Xu(2020)]{10.5555/3495724.3496572}
Denny Wu and Ji~Xu.
\newblock On the optimal weighted $\ell_2$ regularization in overparameterized
  linear regression.
\newblock In \emph{Proceedings of the 34th International Conference on Neural
  Information Processing Systems}, 2020.

\end{thebibliography}
\newpage
\appendix
\section{Proof of theorem 1}\label{subsec:proof_of_lem_bias-variance_decomp}
In this section, we provide the proof of \Cref{thm:asymp_expressions}.
\subsection{Decomposition}
Recall from \Cref{eq:model} that 
We consider the standard linear model:
\begin{equation*}
    Y = \Theta_\ast^{\top} U  + E,
\end{equation*}
where \( \Theta_\ast \in \mathbb{R}^{T \times q} \), \( U := [\mathbf{u_1}, \dots, \mathbf{u_N}] \in \mathbb{R}^{T \times N}\), \( Y :=  [\mathbf{y_1}, \dots, \mathbf{y_N}] \in \mathbb{R}^{q \times N} \), and \( E \in \mathbb{R}^{q \times N} \) has i.i.d. entries with mean zero and variance \( \sigma^2 \). The noise \( E \) is assumed to be independent of \( U \).

We are interested in the generalization performance of Echo State Network (ESN), which we study through its \textit{out-of-sample risk}. For a new, independent test pair \( (\mathbf{u'}, \mathbf{y'}) \), where $\mathbf{\hat{y'}}$ is the prediction, the risk is defined as:
\begin{equation*}
    \mathcal{R} := \frac{1}{q}\mathbb{E}\left[ \left\| \mathbf{y'} - \mathbf{\hat{y'}} \right\|_2^2 \right].
\end{equation*}
where the expectation is taken on everything that is random (from train and test).\\
Substituting the model \( \mathbf{y'} = \Theta_\ast^{\top} \mathbf{u'} + \mathbf{\epsilon'} \), we obtain:
\begin{equation*}
    \mathcal{R} = \frac{1}{q} \mathbb{E}\left[ \left\| \Theta_\ast^{\top} \mathbf{u'} - \mathbf{\hat{y'}} \right\|_2^2 \right] +  \sigma^2.
\end{equation*}
We define the excess risk as  \( \tilde{\mathcal{R}} := \mathcal{R} - \sigma^2 \). Letting \( \Sigma_u := \mathbb{E}[\mathbf{u u}^\top] \in \mathbb{R}^{T \times T} \), we have:
\begin{equation*}
   q\tilde{\mathcal{R}} = \operatorname{Tr}(\Theta_\ast^{\top} \Sigma_u \Theta_\ast) - 2\, \mathbb{E}\left[ \operatorname{Tr}(\mathbf{\hat{y'}  u'}^\top \Theta^{\ast}) \right] + \mathbb{E}\left[ \left\| \mathbf{\hat{y'}} \right\|_2^2 \right].
\end{equation*}

In the ESN framework, predictions are obtained via:
\[
\mathbf{\hat{y'}} = \hat{W}_{\rm out} \mathbf{z'},
\]
where \( \mathbf{z'} \in \mathbb{R}^n \) is the final reservoir state computed from \( \mathbf{u'} \), and the output weights \( \hat{W}_{\rm out} \in \mathbb{R}^{q \times n} \) are estimated using ridge regression: 
\begin{equation*}
    \hat{W}_{\rm out} = \frac{1}{N} Y Z^\top \left( \frac{1}{N} Z Z^\top + \lambda I_n \right)^{-1} ,
\end{equation*}
with \( Z = [\mathbf{z_1}, \dots, \mathbf{z_N}] \in \mathbb{R}^{n \times N} \) and \( Y = [\mathbf{y_1}, \dots, \mathbf{y_N}] \in \mathbb{R}^{q \times N} \).

We define the resolvent $Q:= \left(\frac{1}{N} Z Z^\top + \lambda I_n \right)^{-1}$. The excess risk becomes:
\begin{equation*}
\begin{aligned}
q\tilde{\mathcal{R}} =\ & \operatorname{Tr}(\Theta_\ast^{\top} \Sigma_u \Theta_\ast) 
 - 2\, \mathbb{E} \left[ \operatorname{Tr} \left(  \frac{1}{N} Y Z^\top
Q  \mathbf{z'}   \mathbf{u'}^\top \Theta^{\ast}
\right) \right] 
 + \mathbb{E}\left[
\left\|  \frac{1}{N} Y Z^\top  Q  \mathbf{z'} \right\|_2^2
\right].
\end{aligned}
\end{equation*}

Substituting $Y =  \Theta_\ast^{\top} U + E$ yields:
\begin{equation*}
\begin{aligned}
q\tilde{\mathcal{R}} =\ & \operatorname{Tr}(\Theta_\ast^{\top} \Sigma_u \Theta_\ast)  - 2\, \mathbb{E} \left[ \operatorname{Tr} \left(  \frac{1}{N}  ( \Theta_\ast^{\top} U + E) Z^\top 
Q 
 \mathbf{z'  u'}^\top \Theta^{\ast}
\right) \right]  \\
&+ \mathbb{E} \left[
\left\|  \frac{1}{N} ( \Theta_\ast^{\top} U + E) Z^\top Q
 \mathbf{z'} \right\|_2^2
\right].
\end{aligned}
\end{equation*}

We expand the second and third terms:
\begin{equation*}
\begin{aligned}
q\tilde{\mathcal{R}} =\ & \operatorname{Tr}(\Theta_\ast^{\top} \Sigma_u \Theta_\ast) \\
& - 2\, \mathbb{E} \left[ \operatorname{Tr} \left(  \frac{1}{N}   \Theta_\ast^{\top} U Z^\top 
Q
 \mathbf{z'  u'}^\top \Theta^{\ast}
\right) \right] 
 - 2\, \mathbb{E} \left[ \operatorname{Tr} \left(  \frac{1}{N} E Z^\top
Q
 \mathbf{z'  u'}^\top \Theta^{\ast}
\right) \right] \\
& + \mathbb{E} \left[
\left\|  \frac{1}{N}  \Theta^{\ast \top} U  Z^\top 
Q
 \mathbf{z'} 
\right\|_2^2
\right]  + \mathbb{E} \left[
\left\|  \frac{1}{N} E Z^\top 
Q
 \mathbf{z'} 
\right\|_2^2
\right] \\
& + 2\, \mathbb{E} \left[
\left\langle 
\frac{1}{N}  \Theta^{\ast \top} U  Z^\top 
Q \mathbf{z'},\ 
 \frac{1}{N} E Z^\top  
Q \mathbf{z'} 
\right\rangle
\right]
\end{aligned}
\end{equation*}

By independence of \( E \) and \( \mathbf{u'} \), and zero-mean noise, the cross terms vanish in expectation:
\[
\mathbb{E} \left[ 
\operatorname{Tr} \left( 
 \frac{1}{N}  E Z^\top Q \mathbf{z' u'}^\top  \Theta^{\ast} 
\right) 
\right] = 0, \quad 
\mathbb{E} \left[
\left\langle 
 \frac{1}{N} \Theta_\ast^{\top} U Z^\top Q \mathbf{z'}  
,\ 
\frac{1}{N} E Z^\top  Q \mathbf{z'}
\right\rangle
\right]= 0.
\]

Hence, the excess risk simplifies to:
\begin{equation*}
\begin{aligned}
q\tilde{\mathcal{R}} =\ & \operatorname{Tr}(\Theta^{\ast \top} \Sigma_u \Theta_\ast)  - 2\, \mathbb{E} \left[ \operatorname{Tr} \left( \frac{1}{N} \Theta_\ast^{\top} U Z^\top
Q
 \mathbf{z'  u'}^\top \Theta^{\ast}
\right) \right] \\
& + \mathbb{E} \left[
\left\| \frac{1}{N} \Theta_\ast^{\top} U Z^\top
Q
 \mathbf{z'}
\right\|_2^2
\right] + \mathbb{E} \left[
\left\|  \frac{1}{N} E Z^\top 
Q 
 \mathbf{z'} 
\right\|_2^2
\right].
\end{aligned}
\end{equation*}

Let $\Sigma_{uz} := \mathbb{E}[\mathbf{uz}^\top]$ and $\Sigma_z := \mathbb{E}[\mathbf{zz}^\top]$. We get the decomposition of the excess risk:
\begin{equation*}
\begin{aligned}
q\tilde{\mathcal{R}} =\ 
& \underbrace{
\operatorname{Tr}(\Theta_\ast^{\top} \Sigma_u \Theta_\ast) 
- 2\, \mathbb{E}\left[ \operatorname{Tr} \left(  \frac{1}{N}\Theta^{\ast \top} U  Z^\top 
Q
\Sigma_{uz}^\top \Theta^{\ast}
\right) \right]+
\mathbb{E} \left[
\operatorname{Tr}\left(  \frac{1}{N}\Theta^{\ast \top} U 
Q \Sigma_z   Q U^\top \Theta_\ast 
\right)
\right]
}_{q\mathcal{B}^2} \\
& \quad
+ q \; \underbrace{
\mathbb{E} \left[
\sigma^2 \left(  
 Q \frac{Z^\top Z}{N^2} Q \Sigma_z
\right) 
\right]
}_{\mathcal{V}}.
\end{aligned}
\end{equation*}

\subsection{Asymptotic regime}

\subsubsection{Bias term}
\paragraph{Second Term}
We are interested in finding the limit of $\mathbb{E}\left[ \operatorname{Tr} \left(  \frac{1}{N}\Theta^{\ast \top} U  Z^\top 
Q
\Sigma_{uz}^\top \Theta^{\ast}
\right) \right]$, as both the number of reservoirs $n$ and the number of samples $N$ tend to infinity proportionally (or equivalently, to determine the deterministic equivalent of  $\frac{1}{N} U Z^\top  Q$, in the language of random matrix theory).

This proof follows the structure of arguments in random matrix theory. We provide the complete proof here and refer back to it for the other terms. We begin by disentangling the resolvent $Q$ from the other random variables using Sherman-Morrison's identities (Lemma \ref{lemm:Schur}). We then apply the concentration of quadratic forms to approximate them by their expectations. Finally, we use known deterministic equivalents for the resolvents to conclude. 

\paragraph{Step 1: Disentangling dependence}
Writing $U Z^\top = \sum_{i=1}^N\mathbf{u_i z_i}^\top$ we have:
\begin{equation*}
    \mathbb{E}\left[ \operatorname{Tr} \left(  \frac{1}{N}\Theta^{\ast \top} U  Z^\top 
Q
\Sigma_{uz}^\top \Theta^{\ast}
\right) \right] =  \frac{1}{N}\sum_{i=1}^N
\mathbb{E}\left[ \operatorname{Tr} \left(   \Theta_\ast^{\top} \mathbf{u_i z_i}^\top Q  \Sigma_{uz}^\top \Theta^{\ast}\right) \right]
\end{equation*}
The random variable $(\mathbf{u_i},\mathbf{z_i})$ depends on $Q$, so we use the following lemma to decouple their dependence.
\begin{lemma}\label{lemm:Schur}[Sherman-Morrison Identities]
Let \( Q_{-i} := \left( \frac{1}{N} \sum_{\substack{j=1, j \ne i}}^N \mathbf{z_j z_j}^\top + \lambda I_n \right)^{-1} \) denote the resolvent with the \( i \)-th vector \( \mathbf{z_i} \) removed. Then, the following identities hold:
\begin{equation*}\label{eq:schur1}
    Q = Q_{-i} - \frac{1}{N} \cdot \frac{ Q_{-i} \mathbf{z_i z_i}^\top Q_{-i} }{ 1 + \frac{1}{N} \mathbf{z_i}^\top Q_{-i} \mathbf{z_i} }
\end{equation*}   
\begin{equation*}\label{eq:schur2}
Q\mathbf{z_i} = \frac{Q_{-i}\mathbf{z_i}}{1+\frac{1}{N}\mathbf{z_i}^\top Q_{-i}\mathbf{z_i}}\end{equation*}
\end{lemma}
We use the above equation \ref{eq:schur2}:
\begin{equation*}
\begin{aligned}
    \frac{1}{N}\sum_{i=1}^N
\mathbb{E}\left[ \operatorname{Tr} \left(   \Theta_\ast^{\top} \mathbf{u_i z_i}^\top Q  \Sigma_{uz}^\top \Theta^{\ast}\right) \right] &=  \frac{1}{N}\sum_{i=1}^N
\mathbb{E}\left[ \operatorname{Tr} \left(   \Theta_\ast^{\top} \frac{\mathbf{u_i} \mathbf{z_i}^\top Q_{-i}}{1+\frac{1}{N}\mathbf{z_i}^\top Q_{-i}\mathbf{z_i}} \Sigma_{uz}^\top \Theta^{\ast}\right) \right] 
\end{aligned}
\end{equation*}
We now have that $\mathbf{u_i}$ and $\mathbf{z_i}$ are independent of $Q_{-i}.$
\paragraph{Step 2: Concentration}
The term $\frac{1}{N}\mathbf{z_i}^\top Q_{-i}\mathbf{z_i}$ is still random and depends on $\mathbf{z_i}$, but we will approximate it by its expectation using a concentration argument.

Before proceeding, let us recall some useful lemmas on concentrated random variables and sub-Gaussian random variables which will be used in the proofs.
\begin{definition}[Sub-Gaussian and Sub-exponential random variables]%
\label{def:tail-subg-subexp}
A centered real random variable $x$ is said to be
\emph{$K$-sub-Gaussian} if
\[
      \mathbb{P}\!\bigl(|x|\ge t\bigr)\;\le\;2\exp\!\left(-\frac{t^{2}}{2K^{2}}\right)
      \qquad\forall\,t\ge0.
\]
This definition extends to a random vector $\mathbf{z} \in \mathbb{R}^n$ if, for all unit vector $\mathbf{u}$,$\mathbf{u^\top z}$ is $K$-sub-Gaussian.

A centered real random variable $x$ is called \emph{$(\nu,b)$-sub-exponential} if
\[
      \mathbb{P}\!\bigl(|x|\ge t\bigr)\;\le\;2\exp\!\left[-\,
      \min\!\left(\frac{t^{2}}{2\nu^{2}},\frac{t}{2b}\right)\right]
      \qquad\forall\,t\ge0.
\]
\end{definition}

\begin{lemma}[Basic facts on sub-Gaussian and sub-exponential variables]%
\label{lem:subg-subexp}
Throughout, $\lesssim$ hides an absolute constant. We have the following results.
\begin{enumerate}
  \item[\textup{(i)}] $ \mathbf{z}\in\mathbb{R}^{n}$ is concentrated $\Longrightarrow\; \mathbf{z}$ is sub-Gaussian with a constant independent of $n$.

  \item[\textup{(ii)}] \textbf{Moments of a sub-Gaussian vector.} If $ \mathbf{z}\in\mathbb{R}^{n}$ is $K$-sub-Gaussian, then for every integer $k\ge1$
  \[
      \mathbb{E}\|\mathbf{z}\|_{2}^{k} \;\lesssim\; (K\sqrt n)^{\,k}.
  \]

  \item[\textup{(iii)}] \textbf{Moments of a sub-exponential variable.} Suppose $x  \in \mathbb{R}$ is $(\nu,b)$-sub-exponential with $K=\max(\nu,b)$. Then for every integer $k\ge1$
  \[
      \mathbb{E}|x|^{k} \;\lesssim\; (K k\sqrt n)^{\,k}.
  \]

  \item[\textup{(iv)}] \textbf{Maximum of sub-exponential variables.} Let $\{x_i\}_{i=1}^{m}$ be independent $(\nu,b)$-sub-exponential variables with $K=\max(\nu,b)$. Then for every integer $ k\ge1$
  \[
      \mathbb{E}\left[\max_{1\le i\le m}|x_i|^{k}\right] \;\lesssim\; (K k\log m)^{\,k}.
  \]
\end{enumerate}
\end{lemma}
\begin{proof}
See \cite{Vershynin_2018}.
\end{proof}
\begin{lemma}\label{lem:z_concentrated}
If $\mathbf{z}\in\mathbb{R}^n$ is concentrated (\Cref{def:concentrated}) and $\|\mathbb{E}[\mathbf{z}]\|_2\le M$, then $\Sigma_z:=\mathbb{E}[\mathbf{z}\mathbf{z}^\top]$ satisfies
\[
\|\Sigma_z\| \le K_0 + M^2,
\]
for an absolute constant $K_0>0$ depending only on the concentration constants (hence independent of $n$).
\end{lemma}

\begin{proof}
Write $\mu:=\mathbb{E}[\mathbf{z}]$.
For any $\mathbf{a}\in\mathbb{R}^n$ with $\|\mathbf{a}\|_2=1$,
\[
\mathbf{a}^\top \Sigma_z \mathbf{a}
= \mathbb{E}\big[(\mathbf{a}^\top\mathbf{z})^2\big]
= \mathrm{Var}(\mathbf{a}^\top\mathbf{z}) + (\mathbf{a}^\top\mu)^2.
\]
Since $f(\mathbf{x})=\mathbf{a}^\top\mathbf{x}$ is $1$-Lipschitz, concentration (\Cref{def:concentrated}) gives
\[
\mathbb{P}\big(|\mathbf{a}^\top\mathbf{z}-\mathbf{a}^\top\mu|\ge t\big)\;\le\; C e^{-c t^2}\quad\forall t\ge0.
\]
Hence, by tail integration,
\[
\mathrm{Var}(\mathbf{a}^\top\mathbf{z})
= \mathbb{E}\big[(\mathbf{a}^\top\mathbf{z}-\mathbf{a}^\top\mu)^2\big]
= \int_0^\infty 2t\,\mathbb{P}\big(|\mathbf{a}^\top\mathbf{z}-\mathbf{a}^\top\mu|\ge t\big)\,dt
\le 2C\int_0^\infty t e^{-c t^2}\,dt
= \frac{C}{c}.
\]
Moreover $(\mathbf{a}^\top\mu)^2\le \|\mu\|_2^2\le M^2$. Therefore
\[
\mathbf{a}^\top \Sigma_z \mathbf{a}\;\le\; \frac{C}{c} + M^2
\quad\text{for all }\|\mathbf{a}\|_2=1,
\]
so $\|\Sigma_z\| \le C/c + M^2$. Setting $K_0:=C/c$ completes the proof.
\end{proof}

\begin{lemma}[Concentration of Quadratic forms]\label{lemm:concentration}
    Given a fixed matrix $A$, and $\mathbf{z}$ a concentrated random vector (definition \ref{def:concentrated}) with a bounded expectation $\mathbb{E}[\mathbf{z}]$, we have 
    \begin{equation*}
        \mathbb{E} \left[ \left( \mathbf{z}^\top A \mathbf{y} - \mathbb{E}[\mathbf{z}^\top A \mathbf{y}]\right)^k  \right] \leq C\|A\|_F^k
    \end{equation*}
    for some constant $C>0$.
\end{lemma}

\begin{proof}
Recall the Hanson–Wright inequality for a concentrated random vector $z$ (see  Prop B.41 \cite{louart:tel-04116888}): there exist constants $C',c' > 0$ such that for all $t>0$,
\begin{equation*}
\mathbb{P}\left( |\mathbf{z}^\top Q \mathbf{z} - \mathbb{E}[\mathbf{z}^\top Q \mathbf{z}]| \geq t \right) \leq C'\left( \exp\left( -\frac{c't^2}{\|Q\|_F^2} \right) + \exp\left( -\frac{c't}{\|Q\|_2} \right) \right).
\end{equation*}
Therefore, the random variable $z^\top Q z - \mathbb{E}[z^\top Q z]$ is sub-exponential and from \Cref{lem:subg-subexp} we get the result.
\end{proof}

\begin{lemma}[Operator–norm moment bound]\label{lem:operatornorm}
Let \(A\) be a \(p\times q\) random matrix whose columns \(\mathbf{A_i}\) are
independent, sub‑Gaussian vectors in \(\mathbb{R}^q\) with common
mean \(\boldsymbol{\mu}\) and covariance matrix \(\Sigma\).
If $\|\boldsymbol{\mu}\|_2 = O(\sqrt p)$ and $\| \Sigma\| =O(1)$ then for every integer \(k\ge1\) we have
\[
  \mathbb{E}\bigl[\|A\|^k\bigr] = O(
  \bigl(\sqrt{p}+\sqrt{q}\bigr)^{k}).
\]
\end{lemma}
\begin{proof}
Let \(A':=\bigl(A-\mathbf 1_p\boldsymbol{\mu}^{\!\top}\bigr)\,\Sigma^{-1/2}.\) Then the rows of \(A'\) are i.i.d.\ mean–zero, isotropic sub‑Gaussian vectors. By \cite[Theorem.\,4.6.1]{Vershynin_2018}, for absolute \(c_0,C_0>0\) and all \(t\ge0\),
\begin{equation*}
  \mathbb{P}(\|A'\|> C_0(\sqrt{p}+\sqrt{q})+t\bigr)
  \le 2e^{-c_0t^{2}}.    
\end{equation*}
so \(\|A'\|\) is sub‑exponential with parameters
\(
(\nu,b)=\bigl(C_1(\sqrt{p}+\sqrt{q}),\,C_2\bigr).
\)
The triangle inequality gives
\begin{equation*}
  \|A\|
  \le
  \sqrt{p}\,\|\boldsymbol{\mu}\|_2
  +\;
  \|\Sigma\|^{1/2}\,\|A'\|.  
\end{equation*}

Hence \(\|A\|\) is sub‑exponential with parameters
\begin{equation*}
  \bigl(\widetilde\nu,\widetilde b\bigr)
  \;=\;
  \Bigl(
     \sqrt{p}\,\|\boldsymbol{\mu}\|_2
     +\|\Sigma\|^{1/2}C_1(\sqrt{p}+\sqrt{q}),
     \;
     \|\Sigma\|^{1/2}C_2
  \Bigr)
  \;=\;O\!\bigl(\sqrt{p}+\sqrt{q}\bigr).    
\end{equation*}

Applying \Cref{lem:subg-subexp} we get the desired result.
\end{proof}

Here we proof that we can replace the  quadratic form by it's expectation in our expression.
Let \( \tilde Q_{i} := \mathbb{E}(Q_{-i})  \), $\delta_i := \frac{1}{N} \mathbf{z_i}^\top Q_{-i} \mathbf{z_i}$ and \( \tilde \delta_i := \mathbb{E}\left( \frac{1}{N} \mathbf{z_i}^\top Q_{-i} \mathbf{z_i} \right) = \frac{1}{N} \operatorname{Tr}(\tilde Q_i \Sigma_z) \).
\begin{equation*}\label{eq:difference1}
\begin{aligned}
& \left| \frac{1}{N}\sum_{i=1}^N
\mathbb{E}\left[ \operatorname{Tr} \left( \Theta_\ast^{\top} \frac{\mathbf{\mathbf{u_i z_i}}^\top Q_{-i}}{1+\delta_i}\Sigma_{uz}^\top \Theta^{\ast}   \right) \right]  - \frac{1}{N}\sum_{i=1}^N
\mathbb{E}\left[ \operatorname{Tr} \left( \Theta_\ast^{\top} \frac{ \mathbf{u_i z_i}^\top Q_{-i}}{1+\tilde \delta_i} \Sigma_{uz}^\top \Theta^{\ast}\right) \right] \right| \\ 
&= \left| \frac{1}{N}\sum_{i=1}^N
\mathbb{E}\left[ \operatorname{Tr} \left(   \Theta_\ast^{\top} \mathbf{u_i z_i}^\top Q_{-i}\frac{\delta_i - \tilde \delta_i}{(1+\delta_i)(1+\tilde \delta_i)} \Sigma_{uz}^\top \Theta^{\ast}  \right) \right]   \right| \\
&= \left| \frac{1}{N}\sum_{i=1}^N
\mathbb{E}\left[ \operatorname{Tr} \left(  \Theta_\ast^{\top} \mathbf{u_i z_i}^\top Q\frac{ \delta_i - \tilde \delta_i}{1+\tilde \delta_i} \Sigma_{uz}^\top \Theta^{\ast} \right) \right]   \right| \\
\end{aligned}
\end{equation*}
Let $D := \operatorname{diag}(\frac{\delta_i - \tilde \delta_i}{1+\tilde \delta_i}, \; i \in \{1 \ldots, N \})$, we can write back the sum into a matrix form:
\begin{equation*}
\begin{aligned}
\left| \frac{1}{N}\sum_{i=1}^N
\mathbb{E}\left[ \operatorname{Tr} \left(  \Theta_\ast^{\top} \mathbf{u_i z_i}^\top Q\frac{\delta_i - \tilde \delta_i}{\tilde \delta_i} \Sigma_{uz}^\top \Theta^{\ast} \right) \right]   \right|  & =  \left| \frac{1}{N}
\mathbb{E}\left[ \operatorname{Tr} \left( \Theta_\ast^{\top} U D Z^\top Q \Sigma_{uz}^\top \Theta^{\ast} \right) \right]   \right| \\
& \leq  \frac{1}{N}
\mathbb{E}\left[  \|\Sigma_{uz}\|  \| U\| \|Z\|  \|Q\| \|D\|  \|\Theta^{\ast}  \Theta_\ast^{\top}\|_{F}  \right]  \\
&= \frac{1}{N}  \|\Sigma_{uz}\|  \|\Theta^{\ast}  \Theta_\ast^{\top}\|_{F}
\mathbb{E}\left[  \| U\| \|Z\|  \|Q\| \|D\|   \right]
\end{aligned}
\end{equation*}
We begin by observing that $\|Q\| \leq \frac{1}{\lambda}$, which follows from the fact that $(Z Z^\top / N + \lambda I_n) \succeq \lambda I_n$. Next, we note that $\|\Sigma_{uz}\| = O(1)$. To see this, consider the definition:
\begin{equation*}
\begin{aligned}
    \|\Sigma_{uz}\| &= \sup_{\|\mathbf{a}\|_2 = 1, \|\mathbf{b}\|_2 = 1} \mathbf{a}^\top \Sigma_{uz} \mathbf{b} \\
    &= \sup_{\|\mathbf{a}\|_2 = 1, \|\mathbf{b}\|_2 = 1} \mathbf{a}^\top \mathbb{E}[\mathbf{u z}^\top] \mathbf{b} \\
    &\leq \sup_{\|\mathbf{a}\|_2 = 1, \|\mathbf{b}\|_2 = 1} \sqrt{\mathbb{E}[(\mathbf{a}^\top \mathbf{u})^2]} \sqrt{\mathbb{E}[(\mathbf{z}^\top \mathbf{b})^2]}.
\end{aligned}
\end{equation*}
The first term, $\sqrt{\mathbb{E}[(\mathbf{a}^\top \mathbf{u})^2]}$, does not depend on $n$. For the second term, we have $\mathbb{E}[(\mathbf{z}^\top \mathbf{b})^2] = \mathbf{b}^\top \Sigma_z \mathbf{b} \leq \|\Sigma_z\|$, which is bounded by lemma \ref{lem:z_concentrated}. Hence, $\|\Sigma_{uz}\| = O(1)$ as claimed.  Finally $\|\Theta_\ast \Theta^{\ast^\top}\|_F$ is fixed and deterministic.

That means, we have:
\begin{equation*}
\begin{aligned}
\frac{1}{N}  \|\Sigma_{uz}\|  \|\Theta^{\ast}  \Theta_\ast^{\top}\|_{F}
\mathbb{E}\left[  \| U\| \|Z\|  \|Q\| \|D\|   \right]
=O(   \frac{1}{N} \mathbb{E}\left[  \| U\| \|Z\|  \|D\|   \right])
\end{aligned}
\end{equation*}

Using twice Cauchy-Schwarz, we have
\begin{equation*}
\begin{aligned}
    \frac{1}{N} \mathbb{E}\left[  \| U\| \|Z\|  \|D\|   \right] &\leq \frac{1}{N} \mathbb{E}[  \| U\|^2 \|Z\|^2]^{\frac{1}{2}}  \mathbb{E}[\|D\|^2]^{\frac{1}{2}} \\
    &\leq \frac{1}{N} \mathbb{E}[  \| U\|^4]^{\frac{1}{4}} \mathbb{E}[\|Z\|^4]^{\frac{1}{4}}  \mathbb{E}[\|D\|^2]^{\frac{1}{2}} 
\end{aligned}
\end{equation*}
Moving forward, we apply Lemma~\ref{lem:operatornorm} to obtain bounds on the operator norms of $U$ and $Z$: specifically,
\[
\mathbb{E}[\|U\|^k] = O(\sqrt{N}^k), \quad \text{and} \quad \mathbb{E}[\|Z\|^k] = O((\sqrt{N} + \sqrt{n})^k).
\]
In addition, using Lemma~\ref{lemm:concentration} and using the fact that $\frac{1}{N}\mathbf{z_i}^\top Q_{-i}\mathbf{z_i} \geq 0$ so that $1+\tilde \delta_i \geq 1$, we get
\[
\mathbb{E}[\|D\|^k] = O\left(\frac{\|Q\|_F^k}{N^k}\right) = O\left(\frac{(\sqrt{N} \|Q\|)^k}{N^k}\right) = O\left(\frac{1}{\sqrt{N}^k}\right),
\]
Putting things together, we get that $\frac{1}{N} \mathbb{E}\left[  \| U\| \|Z\|  \|D\|   \right] = O(\frac{1}{\sqrt{N}})$, thus finally we have
\begin{equation*}
 \frac{1}{N}\sum_{i=1}^N
\mathbb{E}\left[ \operatorname{Tr} \left( \Theta_\ast^{\top} \frac{\mathbf{u_i z_i}^\top Q_{-i}}{1+\delta_i} \Sigma_{uz}^\top \Theta^{\ast}   \right) \right]  \to \frac{1}{N}\sum_{i=1}^N
\mathbb{E}\left[ \operatorname{Tr} \left(   \Theta_\ast^{\top} \frac{\mathbf{u_i z_i}^\top Q_{-i}}{1+\tilde \delta_i} \Sigma_{uz}^\top \Theta^{\ast}\right) \right]
\end{equation*}

\paragraph{Step 3: Limit}\label{para:limit}
As a result of step 1 and 2, we can take the expectation of independent random variables:
\begin{equation*}
    \mathbb{E}[ \operatorname{Tr}(\Theta_\ast^{\top} \frac{\mathbf{u_i z_i}^\top Q_{-i}}{1+\tilde \delta_i}\Sigma_{uz}^\top \Theta^{\ast}]) =  \operatorname{Tr}( \Theta_\ast^{\top}  \frac{\Sigma_{uz} \mathbb{E}[Q_{-i}]}{1+\tilde \delta_i} \Sigma_{uz}^\top \Theta^{\ast}),
\end{equation*}
Using Sherman Morrison (lemma \ref{lemm:Schur}) we have $\|Q-Q_{-i}\|_{F} \to 0$. And since $\|\Sigma_{uz}^{\top} \Theta^{\ast}  \Theta_\ast^{\top} \Sigma_{uz} \|  $ is bounded, we have  
\begin{equation*}
    \operatorname{Tr}( \Theta_\ast^{\top}  \frac{\Sigma_{uz} \mathbb{E}[Q_{-i}]}{1+\tilde \delta_i} \Sigma_{uz}^\top \Theta^{\ast})\to  \operatorname{Tr}( \Theta_\ast^{\top}  \frac{\Sigma_{uz} \mathbb{E}[Q]}{1+\tilde \delta} \Sigma_{uz}^\top \Theta^{\ast})  ,
\end{equation*}
The limit of $ \mathbb{E}[Q]$ is a classical result in random matrix theory.
\begin{lemma}[\cite{louart2022concentration}]
Let $\bar Q := (\frac{\Sigma_z}{1+\delta}+\lambda I_n)^{-1}$, we have
\begin{equation*}
    \mathbb{E}[Q - \bar Q]  \to 0
\end{equation*}
with  $\delta$ such as $\tilde \delta  -\delta \to 0 $ and verifying the fixed-point equation:
\begin{equation*}
    \delta - \frac{1}{N} \operatorname{Tr} \left( (\frac{\Sigma_z}{1+\delta}+\lambda I_n)^{-1} \Sigma_z \right) \to 0
\end{equation*}    
\end{lemma}

Given that $\tilde\Sigma \Theta^{\ast}  \Theta_\ast^{\top} \tilde \Sigma^\top$ is bounded in Frobenius norm and using the above limit, we have finally 
\begin{equation*}
    \mathbb{E}\left[ \operatorname{Tr} \left( \left( \frac{1}{N} Z^\top U \Theta_* \right)^\top Q \tilde\Sigma \Theta^{*}\right) \right] -\operatorname{Tr} \left(  \tilde\Sigma \Theta_\ast \Theta_\ast^{\top}\frac{\tilde\Sigma^\top \bar Q}{1+\delta}\right) \to 0
\end{equation*}
\paragraph{Third Term}
Let's focus on determining the limit of 
\begin{equation*}
\begin{aligned}
    \mathbb{E} \left[
\left\| \left( \frac{1}{N} Z^\top U \Theta_\ast \right)^\top
\left( \frac{1}{N} Z^\top Z + \lambda I_n \right)^{-1} 
 \mathbf{z_0} 
\right\|_2^2 \right] 
&=
\frac{1}{N^2}\mathbb{E} \left[ \Theta_\ast^{\top} U^\top Z 
Q
\mathbf{z_0 z_0}^\top
Q
 Z^\top U \Theta_\ast\right] \\
&= \frac{1}{N^2}\mathbb{E} \left[ \Theta_\ast^{\top} U^\top Z 
Q
\Sigma_z
Q
 Z^\top U \Theta_\ast\right]
\end{aligned}
\end{equation*}
Let's decouple $Q$ from $U$ and $Z$ as we did above. We begin by writing:
\begin{equation*}
\begin{aligned}
\frac{1}{N^2}\mathbb{E} \left[ \Theta_\ast^{\top} U^\top Z 
Q
\Sigma_z
Q
 Z^\top U \Theta_\ast\right]=\frac{1}{N^2} \sum_{i=1}^N\mathbb{E} \left[ \Theta_\ast^{\top} \mathbf{u_i z_i}^\top 
Q
\Sigma_z
Q
 Z^\top U \Theta_\ast\right]
\end{aligned}
\end{equation*}

Using Sherman-Morrison identity (equation \ref{eq:schur2}), we have
\begin{equation*}
\begin{aligned}
\frac{1}{N^2} \sum_{i=1}^N\mathbb{E} \left[ \Theta_\ast^{\top} \mathbf{u_i z_i}^\top 
Q
\Sigma_z
Q
 Z^\top U \Theta_\ast\right] = \frac{1}{N^2}\sum_{i=1}^N\mathbb{E} \left[ \Theta_\ast^{\top} \mathbf{u_i z_i}^\top 
\frac{Q_{-i}}{1+ \delta_i}
\Sigma_z
Q
 Z^\top U \Theta_\ast\right]
\end{aligned}
\end{equation*}
We use now the fact that $\delta_i$ is concentrated around its expectation $\tilde \delta_i$
\begin{equation*}
\begin{aligned}
\frac{1}{N^2}\sum_{i=1}^N\mathbb{E} \left[ \Theta_\ast^{\top} \mathbf{u_i z_i}^\top 
\frac{Q_{-i}}{1+ \delta_i}
\Sigma_z
Q
 Z^\top U \Theta_\ast\right] - \frac{1}{N^2}\sum_{i=1}^N\mathbb{E} \left[ \Theta_\ast^{\top} \mathbf{u_i z_i}^\top 
\frac{Q_{-i}}{1+\tilde \delta_i}
\Sigma_z
Q
 Z^\top U \Theta_\ast\right] \to 0
\end{aligned}
\end{equation*}
To justify that, we rewrite the difference between the two above sums as $\frac{1}{N^2} \mathbb{E}[\Theta_\ast^{\top} U^\top \Delta_1 Z 
Q
\Sigma_z
Q
 Z^\top U \Theta_\ast]$ where $\Delta_1 = \operatorname{diag}(\frac{\delta_i-\tilde\delta_i}{1+\tilde \delta_i}, i \in \{ 1,\ldots,N \})$ and  $\mathbb{E}[\| \Delta_1\|^k] = O((\frac{logN}{\sqrt{N}})^k), k \in \mathbb{N}$. [Justification max of subgaussians is log N (\Cref{lem:subg-subexp})
We now do the same thing on the other side
\begin{equation*}
\begin{aligned}
\frac{1}{N^2}\sum_{i=1}^N\mathbb{E} \left[ \Theta_\ast^{\top} \mathbf{u_i z_i}^\top 
\frac{Q_{-i}}{1+\tilde \delta_i}
\Sigma_z
Q
 Z^\top U \Theta_\ast\right] 
 &=\frac{1}{N^2}\sum_{i,j=1}^N\mathbb{E} \left[ \Theta_\ast^{\top} u_i \mathbf{z_i}^\top 
\frac{Q_{-i}}{1+\tilde \delta_i}
\Sigma_z
Q
 \mathbf{z_j u_j}^\top \Theta_\ast\right] \\
 &=\frac{1}{N^2}\sum_{i,j=1}^N\mathbb{E} \left[ \Theta_\ast^{\top} \mathbf{u_i z_i}^\top 
\frac{Q_{-i}}{1+\tilde \delta_i}
\Sigma_z
\frac{Q_{-j}}{1+\delta_j}
 \mathbf{z_j u_j}^\top \Theta_\ast\right] \\
 & \to \frac{1}{N^2}\sum_{i,j=1}^N\mathbb{E} \left[ \Theta_\ast^{\top} u_i \mathbf{z_i}^\top 
\frac{Q_{-i}}{1+\tilde \delta_i}
\Sigma_z
\frac{Q_{-j}}{1+\tilde \delta_j}
 \mathbf{z_j u_j}^\top \Theta_\ast\right]
\end{aligned}
\end{equation*}
We justify the last step by writing the difference between the last two terms as $\frac{1}{N^2}\mathbb{E} \left[ \Theta_\ast^{\top} U^\top \Delta_2 Z 
Q
\Sigma_z
Q Z^T \Delta_1 U \Theta_\ast\right]$, where $\Delta_2 = \operatorname{diag}(\frac{1+\delta_i}{1+\tilde  \delta_i}, i \in \{ 1,\ldots,N \})$ and  $\mathbb{E}[\| \Delta_2\|^k] = O(1), k \in \mathbb{N}$. 

We split now the sum into two terms $i=j$ and $i \neq j$,
\begin{equation*}
\begin{aligned}
 \frac{1}{N^2}\sum_{i,j=1}^N\mathbb{E} \left[ \Theta_\ast^{\top} \mathbf{u_i z_i}^\top 
\frac{Q_{-i}}{1+\tilde \delta_i}
\Sigma_z
\frac{Q_{-j}}{1+\tilde \delta_j}
 \mathbf{z_j u_j}^\top \Theta_\ast\right] &=  \frac{1}{N^2}\sum_{i=1}^N\mathbb{E} \left[ \Theta_\ast^{\top} \mathbf{u_i z_i}^\top 
\frac{Q_{-i}}{1+\tilde \delta_i}
\Sigma_z
\frac{Q_{-i}}{1+\tilde \delta_i}
 \mathbf{z_i u_i}^\top \Theta_\ast\right] \\
 &+ \frac{1}{N^2}\sum_{\substack{i,j=1 \\ i \ne j}}^N\mathbb{E} \left[ \Theta_\ast^{\top} \mathbf{u_i z_i}^\top 
\frac{Q_{-i}}{1+\tilde \delta_i}
\Sigma_z
\frac{Q_{-j}}{1+\tilde \delta_j}
 \mathbf{z_j u_j}^\top \Theta_\ast\right]
\end{aligned}
\end{equation*}

For the term $i=j$, we use the concentration of the quadratic form random variable $\xi_i :=\frac{1}{N}\mathbf{z_i}^\top Q_{-i}\Sigma_z Q_{-i}\mathbf{z_i}$ around its expectation $\tilde \xi_i$,
\begin{equation*}
\begin{aligned}
  \frac{1}{N^2}\sum_{i=1}^N\mathbb{E} \left[ \Theta_\ast^{\top} \mathbf{u_i z_i}^\top 
\frac{Q_{-i}}{1+\tilde \delta_i}
\Sigma_z
\frac{Q_{-i}}{1+\tilde \delta_i}
 \mathbf{z_i u_i}^\top \Theta_\ast\right] 
 -  \frac{1}{N}\sum_{i=1}^N \frac{\tilde \xi_i}{(1+\tilde \delta_i)^2} \Theta_\ast^{\top}\Sigma_u\Theta_\ast \to 0
\end{aligned}
\end{equation*}
The justification is the same as above, we write the difference as $\frac{1}{N}\sum_{i=1}^N\mathbb{E} \left[ \Theta_\ast^{\top} U^\top \Delta_3 U \Theta_\ast\right] $, where $\Delta_3 = \operatorname{diag}(\frac{\xi_i - \tilde \xi_i }{(1+\tilde \delta_i)^2}, i \in \{ 1,\ldots,N \})$ and  $\mathbb{E}[\| \Delta_3\|^k] = O((\frac{logN}{\sqrt{N}})^k), k \in \mathbb{N}$.

For the term $i \neq j $, we still have dependency between $Q_{-i}$ and $Q_{-j}$, so we reapply Sherman-Morrison (equation \ref{eq:schur1} this time):
\begin{equation*}
\begin{aligned}
&\frac{1}{N^2}\sum_{\substack{i,j=1 \\ i \ne j}}^N\mathbb{E} \left[ \Theta_\ast^{\top} \mathbf{u_i z_i}^\top 
\frac{Q_{-i}}{1+\tilde \delta_i}
\Sigma_z
\frac{Q_{-j}}{1+\tilde \delta_j}
 \mathbf{z_j u_j}^\top \Theta_\ast\right] = \kappa_1 + \kappa_2 + \kappa_3 + \kappa_4
\end{aligned}
\end{equation*}
with
\begin{equation*}
\begin{aligned}
&\kappa_1= \frac{1}{N^2}\sum_{\substack{i,j=1 \\ i \ne j}}^N\mathbb{E} \left[ \Theta_\ast^{\top} \mathbf{u_i z_i}^\top 
\frac{Q_{-i,j}}{1+\tilde \delta_i}
\Sigma_z
\frac{Q_{-i,j}}{1+\tilde \delta_j}
 \mathbf{z_j u_j}^\top \Theta_\ast\right]\\
&\kappa_2 = - \frac{1}{N^2}\sum_{\substack{i,j=1 \\ i \ne j}}^N \frac{1}{N}\mathbb{E} \left[ \Theta_\ast^{\top} \mathbf{u_i z_i}^\top 
\frac{Q_{-i} \mathbf{z_j}\mathbf{z_j^\top} Q_{-i,j}}{1+\tilde \delta_i}
\Sigma_z
\frac{Q_{-j}}{1+\tilde \delta_j}
 \mathbf{z_j}\mathbf{u_j}^\top \Theta_\ast\right] \\
&\kappa_3 = - \frac{1}{N^2}\sum_{\substack{i,j=1 \\ i \ne j}}^N \frac{1}{N} \mathbb{E} \left[ \Theta_\ast^{\top} \mathbf{u_i}\mathbf{z_i}^\top 
\frac{Q_{-i}}{1+\tilde \delta_i}
\Sigma_z
\frac{Q_{-j}\mathbf{z_i}\mathbf{z_i}^\top Q_{-i,j}}{1+\tilde \delta_j}
 \mathbf{z_j} \mathbf{u_j}^\top \Theta_\ast\right]  \\
&\kappa_4 = \frac{1}{N^2}\sum_{\substack{i,j=1 \\ i \ne j}}^N \frac{1}{N^2} 
\mathbb{E} \left[ \Theta_\ast^{\top} \mathbf{u_i} \mathbf{z_i}^\top 
\frac{Q_{-i} \mathbf{z_j} \mathbf{z_j}^\top Q_{-i,j}}{1+\tilde \delta_i}
\Sigma_z
\frac{Q_{-j} \mathbf{z_i} \mathbf{z_i}^\top Q_{-i,j}}{1+\tilde \delta_j}
 \mathbf{z_j} \mathbf{u_j}^\top \Theta_\ast\right]
\end{aligned}  
\end{equation*}
It is now straightforward that 
\begin{equation*}
    \kappa_1 = \frac{1}{N^2}\sum_{\substack{i,j=1 \\ i \ne j}}^N \frac{ \Theta_\ast^{\top} \Sigma_{uz} 
\mathbb{E}[Q_{-i,j}
\Sigma_z
Q_{-i,j}]
 \Sigma_{uz}^\top \Theta_\ast}{(1+\tilde \delta_i)(1+\tilde \delta_j)}
\end{equation*}
For $\kappa_2$, we will redo Sherman-Morisson (equation \ref{eq:schur1})
\begin{equation*}
\begin{aligned}
   & \frac{1}{N^2}\sum_{\substack{i,j=1 \\ i \ne j}}^N \frac{1}{N}\mathbb{E} \left[ \Theta_\ast^{\top} \mathbf{u_i} \mathbf{z_i}^\top 
\frac{Q_{-i}\mathbf{z_j}\mathbf{z_j}^\top Q_{-i,j}}{1+\tilde \delta_i}
\Sigma_z
\frac{Q_{-j}}{1+\tilde \delta_j}
 \mathbf{z_j} \mathbf{u_j}^\top \Theta_\ast\right] \\
 &  = \frac{1}{N^2}\sum_{\substack{i,j=1 \\ i \ne j}}^N \frac{1}{N}\mathbb{E} \left[ \Theta_\ast^{\top} \mathbf{u_i} \mathbf{z_i}^\top 
\frac{Q_{-i} \mathbf{z_j}\mathbf{z_j}^\top Q_{-i,j}}{1+\tilde \delta_i}
\Sigma_z
\frac{Q_{-i,j}}{1+\tilde \delta_j}
 \mathbf{z_j} \mathbf{u_j}^\top \Theta_\ast\right] \\
 & - \frac{1}{N^2}\sum_{\substack{i,j=1 \\ i \ne j}}^N \frac{1}{N^2}\mathbb{E} \left[ \Theta_\ast^{\top} \mathbf{u_i} \mathbf{z_i}^\top 
\frac{Q_{-i,j} \mathbf{z_j} \mathbf{z_j}^\top Q_{-i,j}}{(1+\tilde \delta_i)(1+\frac{1}{N} \mathbf{z_j}^\top Q_{-i,j}\mathbf{z_j})}
\Sigma_z
\frac{Q_{-i,j}\mathbf{z_i} \mathbf{z_i}^\top Q_{-i,j}}{(1+\tilde \delta_j)(1+\frac{1}{N}\mathbf{z_i}^\top Q_{-i,j}\mathbf{z_i})}
 \mathbf{z_j} \mathbf{u_j}^\top \Theta_\ast\right]
\end{aligned}
\end{equation*}
The second term is negligible. In fact using Hanson Wright and the fact that $\|\boldsymbol{\mu_z}\|_2 =O(1)$ ($\mathbb{E}[\mathbf{z_i}A\mathbf{z_j}] = \boldsymbol{\mu_z} A \boldsymbol{\mu_z}, \; i\neq j$ with $\|A\|=O(1))$ , we have $\mathbb{E}[|\mathbf{z_i}A\mathbf{z_j}|^k] =O(\sqrt{N}^k)$. We have 3 random variables of that type, so using Cauchy-Schwarz, we get that the second term is $O(\frac{\sqrt{N}^3}{N^2}) = o(1)$.

We thus have 
\begin{equation*}
    \kappa_2 - (- \frac{1}{N^2}\sum_{\substack{i,j=1 \\ i \ne j}}^N \frac{1}{N}\mathbb{E} \left[ \Theta_\ast^{\top} \mathbf{u_i} \mathbf{z_i}^\top 
\frac{Q_{-i}\mathbf{z_j}\mathbf{z_j}^\top Q_{-i,j}}{1+\tilde \delta_i}
\Sigma_z
\frac{Q_{-i,j}}{1+\tilde \delta_j}
 \mathbf{z_j} \mathbf{u_j}^\top \Theta_\ast\right] ) \to 0
\end{equation*}
We then use the concentration of the random variable $\eta_{i,j} = \frac{1}{N} \mathbf{z_j} Q_{-i,-j}\Sigma_z Q_{-i,j} \mathbf{z_j}$ around its expectation $\tilde \eta_{i,j}$. To justify it, we take the difference
\begin{equation*}
\begin{aligned}
&\frac{1}{N^2}\sum_{\substack{i,j=1 \\ i \ne j}}^N \frac{1}{N}\mathbb{E} \left[ \Theta_\ast^{\top} \mathbf{u_i} \mathbf{z_i}^\top 
Q_{-i,j}z_j \frac{\eta_{i,j}-\tilde \eta_{i,j}}{(1+\tilde \delta_i)(1+\delta_{i,j})(1+\tilde \delta_j)}\mathbf{u_j}^\top \Theta_\ast\right] \\
&= \frac{1}{N^2} \sum_{i=1}^N \mathbb{E} \left[ \Theta_\ast^{\top} \mathbf{u_i} \mathbf{z_i}^\top 
Q_{-i} \sum_{\substack{j=1 \\ i \ne j}}^N \mathbf{z_j} \frac{\eta_{i,j}-\tilde \eta_{i,j}}{(1+\tilde \delta_i)(1+\tilde \delta_j)}u_j^\top \Theta_\ast\right]\\
&= \frac{1}{N^2} \sum_{i=1}^N \mathbb{E} \left[ \Theta_\ast^{\top} \Sigma_{uz} 
Q_{-i} \sum_{\substack{j=1 \\ i \ne j}}^N \mathbf{z_j} \frac{\eta_{i,j}-\tilde \eta_{i,j}}{(1+\tilde \delta_i)(1+\tilde \delta_j)}\mathbf{u_j}^\top \Theta_\ast\right]\\
&= \frac{1}{N^2} \sum_{i=1}^N \mathbb{E} \left[ \Theta_\ast^{\top} \Sigma_{uz}
Q_{-i} Z^\top \Delta_i U \Theta_\ast\right]
\end{aligned}
\end{equation*}
where $\Delta_i = \operatorname{diag}(\frac{\eta_{i,j}-\tilde \eta_{i,j}}{(1+\tilde \delta_i)(1+\tilde \delta_j)}, j \in \{ 1,\ldots,N \} \backslash\{i\} \text{ and } 0 \text{ for $i=j$})$ and  $\mathbb{E}[\| \Delta_i\|^k] = O((\frac{logN}{\sqrt{N}})^k), k \in \mathbb{N}$. This means that
\begin{equation*}
    \kappa_2 - ( - \frac{1}{N^2}\sum_{\substack{i,j=1 \\ i \ne j}}^N \frac{\tilde \eta_{i,j}}{(1+\tilde \delta_i)(1+\tilde \delta_j)}\mathbb{E} \left[ \Theta_\ast^{\top} \mathbf{u_i} \mathbf{z_i}^\top 
Q_{-i}\mathbf{z_j}  \mathbf{u_j}^\top \Theta_\ast\right] ) \to 0
\end{equation*}
As we did above, this gives as that 
\begin{equation*}
\begin{aligned}
    \kappa_2-
(- \frac{1}{N^2}\sum_{\substack{i,j=1 \\ i \ne j}}^N \frac{\tilde \eta_{i,j}}{(1+\tilde \delta_i)(1+\tilde \delta_j)(1+\tilde \delta_{i,j})} \Theta_\ast^{\top} \Sigma_{uz} 
\mathbb{E}[Q_{-i,j}]\Sigma_{uz}^\top \Theta_\ast)\to 0
\end{aligned}
\end{equation*}

$\kappa_3$ is the same technique and $\kappa_4$ is negligible as we did above.

The limit of $\mathbb{E}[Q\Sigma_zQ]$ is also a classical result in random matrix theory :
\begin{lemma}[\cite{couillet2022random}]\label{lem:secondorder}
Let $\mathcal{Q} := \frac{(1+\delta)^2}{(1+\delta)^2-\frac{1}{N}\operatorname{Tr}(\Sigma_z\bar Q\Sigma_z\bar Q)}\bar Q\Sigma_z\bar Q$, we have
    \begin{equation*}
        \mathbb{E}[Q\Sigma_zQ - \mathcal{Q}] \to 0.  
    \end{equation*}
\end{lemma}
Now as we did above in the step 3 limit \ref{para:limit}, we can replace everything with its limit, then we get the desired result after simplifications.
\subsection{Variance Term}
In the same manner, we want the limit of
\begin{equation*}
    \mathbb{E} \left[
 \left(  
z_0^\top \left( \frac{Z^\top Z}{N} + \lambda I_n \right)^{-1} \frac{Z^\top Z}{N^2} \left( \frac{Z^\top Z}{N} + \lambda I_n \right)^{-1} z_0
\right) 
\right] = \mathbb{E} \left[ \operatorname{Tr} \left( Q \frac{Z^\top Z}{N^2} Q \Sigma_z \right)  \right],
\end{equation*}
We write:
\begin{equation*}
\begin{aligned}
    \mathbb{E} \left[  \operatorname{Tr} \left( Q \frac{Z^\top Z}{N^2} Q \Sigma_z  \right) \right] &= \frac{1}{N^2}\sum_{i=1}^N \mathbb{E} \left[  \operatorname{Tr} \left( Q \mathbf{z_i} \mathbf{z_i}^\top Q \Sigma_z \right)  \right] \\
    &= \frac{1}{N^2}\sum_{i=1}^N \mathbb{E} \left[ \operatorname{Tr} \left( \frac{Q_{-i}}{1+\frac{1}{N}\mathbf{z_i}^\top Q_{-i}\mathbf{z_i}} \mathbf{z_i} \mathbf{z_i}^\top \frac{Q_{-i}}{1+\frac{1}{N}\mathbf{z_i}^\top Q_{-i}\mathbf{z_i}} \Sigma_z \right)  \right]
\end{aligned}
\end{equation*}
As we did for the second term of the bias, we can here replace $\frac{1}{N}\mathbf{z_i}^\top Q_{-i} \mathbf{z_i}$ with its expectation in the limit. Then using \Cref{lem:secondorder}, we get the desired result after simplification.

\section{Proof of theorem 2}\label{proof:th2}
Let  $\rho':=\varphi \cdot\rho(W_0)$, we have
\[
S = \begin{bmatrix}
\left(\tfrac{W_0}{\rho'}\right)^{T-1} \mathbf{w}_{\text{in}} ,
\left(\tfrac{W_0}{\rho'}\right)^{T-2} \mathbf{w}_{\text{in}}  ,
\ldots ,
\left(\tfrac{W_0}{\rho'}\right)^0 \mathbf{w}_{\text{in}}
\end{bmatrix}  \in \mathbb{R}^{n \times T}
\]
denote the state matrix built from the input weight vector \(\mathbf{w}_{\rm in} \in \mathbb{R}^n\) and recurrent matrix \(W \in \mathbb{R}^{n \times n}\). We get $Z = SU$.

\begin{lemma}[Concentration of $S^\top S$]
Let \( W_0 \in \mathbb{R}^{n \times n} \) have i.i.d.\ standard Gaussian entries and let \( \mathbf{w}_{\mathrm{in}} \sim \mathcal{N}(0, \frac{1}{n}I_n) \)
and \( \rho' = \varphi \cdot \rho(W_0) \) with fixed \( \varphi < 1 \), and let \( S \in \mathbb{R}^{T \times n} \) be defined as in \ref{eq:S}. Then, for some constant \( C > 0 \) depending only on \( T \) and \( \varphi \), we have
\[
\mathbb{P}\left( \max_{i,j \leq T} \left| [S^\top S]_{ij} - \mathbb{E}[S^\top S]_{ij} \right| > \frac{C}{\sqrt{n}} \right) \leq c_1 e^{-c_2 n}
\]
for some constants \( c_1, c_2 > 0 \). In particular,
\[
\| S^\top S - \mathbb{E}[S^\top S] \|_F = O_{\mathbb{P}} \left( \frac{1}{\sqrt{n}} \right).
\]
\end{lemma}

\begin{proof}
For each \( i, j \in \{0, \dots, T-1\} \), define
\[
A_{ij} := \left( \frac{W_0^\top}{\rho'} \right)^i \left( \frac{W_0}{\rho'} \right)^j \in \mathbb{R}^{n \times n},
\]
so that
\[
[S^\top S]_{ij} = \langle S_i, S_j \rangle = \mathbf{w}_{\mathrm{in}}^\top A_{ij} \mathbf{w}_{\mathrm{in}}.
\]

From \cite{Alt_2021} we have almost surely as $n$ tends to infinity $\rho(W_0)/\sqrt n \to 1$, and from Bai-Yin Theorem (\cite{noauthororeditor} + Remark on i.i.d. standard random matrices), we have almost surely, $\|W_0\|/\sqrt n \to 2$. 
Fix $\varepsilon \in (0,1)$. Then, almost surely for all sufficiently large $n$,
\[
\rho(W_0) \ge (1-\varepsilon)\sqrt{n},
\qquad
\|W_0\| \le (2+\varepsilon)\sqrt{n}.
\]
Since $\rho' = \varphi \rho(W_0)$ with fixed $\varphi \in (0,1)$, it follows that
\[
\Big\|\frac{W_0}{\rho'}\Big\|
= \frac{\|W_0\|}{\varphi \rho(W_0)}
\le \frac{2+\varepsilon}{\varphi(1-\varepsilon)}
=: C_\varepsilon,
\]
for all large $n$. In particular, for each $i,j \leq T-1$,
\[
\|A_{ij}\| 
= \Big\|\Big(\frac{W_0^\top}{\rho'}\Big)^i \Big(\frac{W_0}{\rho'}\Big)^j\Big\|
\le \|W_0/\rho'\|^{\,i+j}
\le C_\varepsilon^{\,2(T-1)} =: M,
\]
and hence $\|A_{ij}\|_F \le \sqrt{n}\,M$.

Now write $X := \sqrt{n}\,\mathbf{w}_{\mathrm{in}} \sim \mathcal N(0,I_n)$. Then
\[
[S^\top S]_{ij} - \mathbb{E}[S^\top S]_{ij}
= \frac{1}{n}\Big(X^\top A_{ij}X - \mathbb{E}[X^\top A_{ij}X]\Big).
\]
By the Hanson--Wright inequality ($W_0$ independent of $X$),
\[
\mathbb{P}\!\left(\Big|X^\top A_{ij}X - \mathbb{E}[X^\top A_{ij}X]\Big| > u \right)
\le 2 \exp\!\left(-c \min\Big\{\tfrac{u^2}{\|A_{ij}\|_F^2},\,\tfrac{u}{\|A_{ij}\|}\Big\}\right).
\]
Taking $u = n t$ and using $\|A_{ij}\|\le M$, $\|A_{ij}\|_F \le \sqrt{n}M$ gives
\[
\mathbb{P}\!\left(\big|[S^\top S]_{ij} - \mathbb{E}[S^\top S]_{ij}\big| > t\right)
\le 2 \exp\!\left(-c \min\Big\{\tfrac{n t^2}{M^2},\,\tfrac{n t}{M}\Big\}\right).
\]
In particular, for $t = x/\sqrt{n}$,
\[
\mathbb{P}\!\left(\big|[S^\top S]_{ij} - \mathbb{E}[S^\top S]_{ij}\big| > \tfrac{x}{\sqrt{n}} \right)
\le 2 \exp\!\left(-c \,\frac{x^2}{M^2}\right).
\]
A union bound over all $i,j\le T-1$ yields
\[
\mathbb{P}\!\left(\max_{i,j} \big|[S^\top S]_{ij} - \mathbb{E}[S^\top S]_{ij}\big|
> \tfrac{x}{\sqrt{n}}\right)
\le 2 T^2 \exp\!\left(-c \,\frac{x^2}{M^2}\right),
\]
\end{proof}

\begin{lemma}[Limit of \ensuremath{\mathbb{E}[S^\top S]}]
As $T$ is fixed, we have as \( n \to \infty \), 
\[
\mathbb{E}[S^{\!\top}S]
\longrightarrow
\operatorname{diag}\left(\varphi^{-(T-1)},\,\varphi^{-(T-2)},\,\dots,\,\varphi^{-1},\,1\right) \quad \] 
\end{lemma}
\begin{proof}
We have
\begin{equation*}
[S^\top S]_{i,j} = \left\langle \left(\tfrac{W_0}{\rho'}\right)^{T-i} \mathbf{w}_{\text{in}}, \left(\tfrac{W_0}{\rho'}\right)^{T-j} \mathbf{w}_{\text{in}} \right\rangle, \quad 1 \leq i,j \leq T
\end{equation*}

Because $\bm w_{\mathrm{in}}$ is independent of $W_0$, we get
\[
\mathbb E_{\mathbf{w_{\rm in}}}\bigl[S^{\!\top}S\bigr]_{ij}
   = \frac{1}{n}(\rho')^{-2T+i+j}\,
        \operatorname{Tr}\bigl(W_0^{\top\,(T-i)}W_0^{\,T-j}\bigr).
\]
Since $W_0$ has i.i.d. $\mathcal{N}(0,1)$, we have, using \cite{kemp2009enumerationnoncrossingpairingsbit} [Prop + Remark 1.4, Prop 1.7], that
\[
\frac{1}{n}
     \operatorname{Tr}\bigl((\frac{W_0}{\sqrt n})^{\top\,(T-i)}(\frac{W_0}{\sqrt n})^{\,T-j}\bigr)
   \xrightarrow{n\to\infty}\delta_{ij} \quad\text{a.s.}
\]

From \cite{Alt_2021}, we have
\(
\rho(W_0)/\sqrt n\to 1
\)
almost surely.

We insert this back to get:
\[
\mathbb E_{\mathbf{w_{\rm in}}}\bigl[S^{\!\top}S\bigr]_{ij} \xrightarrow{n\to\infty} \delta_{ij} \varphi^{-(T-i)}\quad\text{a.s.}
\]
Therefore:
\begin{equation*}
\mathbb E_{\mathbf{w_{\rm in}}}\bigl[S^{\!\top}S\bigr] \xrightarrow{n\to\infty}\operatorname{diag}\bigl(\varphi^{-(T-1)},\;\varphi^{-(T-2)},\;\dots,\;\varphi^{-1},\;1\bigr)\quad\text{a.s.}
\end{equation*}

For every $n$, we have $\mathbb E_{\mathbf{w_{\rm in}}}\bigl[S^{\!\top}S\bigr]_{ij}$ is almost surely bounded. In particular it is uniformly integral, that is for any \( K > 1 \),
\[
\sup_n \mathbb{E}\left[|\mathbb E_{\mathbf{w_{\rm in}}}\bigl[S^{\!\top}S\bigr]_{ij}| \cdot \mathbf{1}_{\{|\mathbb E_{\mathbf{w_{\rm in}}}\bigl[S^{\!\top}S\bigr]_{ij}| > K\}}\right] = 0,
\]
Thus we have the convergence in expectation, that is
\[
\mathbb{E}[S^{\!\top}S]
\longrightarrow
\operatorname{diag}\left(\varphi^{-(T-1)},\,\varphi^{-(T-2)},\,\dots,\,\varphi^{-1},\,1\right) \quad \] 
\end{proof}
Now that we have established the limit of $S^\top S$, we can proceed to compute the limit of the risk. As mentioned in \Cref{rm:thm1}, under the hypothesis that $\mathbf{u}$ is concentrated, the vector $\mathbf{z}$,obtained as a Lipschitz function of $\mathbf{u}$ with a bounded Lipschitz constant, is also concentrated. We can thus apply \Cref{thm:asymp_expressions}, let $\tilde S := S\Sigma_u^\frac{1}{2}$ and $\tilde \Theta_\ast = \Sigma_u^\frac{1}{2} \Theta_\ast  $ We have the asymptotic bias is given by:
\begin{equation*}
  \mathcal{B}^{2}
  \;\longrightarrow\;
  \mathcal{B}_\infty^{2}
  \;:=\;
\frac{1}{1 - \alpha} \Bigg(
    \operatorname{Tr}\!\bigl(\tilde\Theta^{\ast\!\top}\tilde\Theta^{\ast}\bigr) 
    - \frac{2}{1 + \delta} \operatorname{Tr}\!\Bigl(
      \tilde\Theta^{\ast\!\top}
      [\tilde S^{\!\top}\bar{Q}\tilde S]
      \tilde\Theta^{\ast}
    \Bigr)
     + \frac{1}{(1 + \delta)^2} \operatorname{Tr}\!\Bigl(
      \tilde\Theta^{\ast\!\top}
      [\tilde S^{\!\top}\bar{Q}\tilde S]^{2}
      \tilde\Theta^{\ast}
    \Bigr)
\Bigg)
\end{equation*}
where 
\begin{align*}
    \bar{Q} &:= \left( \frac{\tilde S\tilde S^{\!\top}}{1+\delta} + \lambda I_n \right)^{-1},
    \quad 
    \delta = \frac{1}{N}\operatorname{Tr}\!\bigl[\tilde S^{\!\top}\bar{Q}\tilde S\bigr], \quad \alpha = \frac{\operatorname{Tr}([\tilde S^{\!\top}\bar{Q}\tilde S]^{2})}{N(1+\delta)^2}
\end{align*}
Since we know the expectation of \( \tilde{S}^\top \tilde{S} \in \mathbb{R}^{T \times T} \), we will use the Woodbury identity to get it in the expression instead of $\tilde S\tilde S^{\!\top}$:
\begin{equation*}
  \tilde{S}^{\!\top}\bar{Q}\tilde{S} 
  = \tilde{S}^\top \tilde{S} \left( \frac{\tilde{S}^\top \tilde{S}}{1+\delta} + \lambda I_T \right)^{-1}
\end{equation*}
That makes \( \mathcal{B}_\infty^2 \) a function of \( \tilde{S}^\top \tilde{S} \). Our goal is to use the fact that the map
$\tilde{S}^\top \tilde{S} \mapsto \mathcal{B}_\infty^2(\tilde{S}^\top \tilde{S})$
is Lipschitz in Frobenius norm with Lipschitz constant $L=O(1)$, so that we may formally replace \( \tilde{S}^\top \tilde{S} \) by the limit of its expectation inside the global expectation. That is, we want (by defining $M_\infty:= \Sigma_u^{\frac{1}{2}}\operatorname{diag}(\varphi^{(i-T)})\Sigma_u^{\frac{1}{2}}$):

\begin{equation*}
\begin{aligned}
    \left| \mathcal{B}_\infty^2 (\tilde{S}^\top \tilde{S}) - \mathcal{B}_\infty^2 (M_\infty) \right| 
    &\leq \left| \mathcal{B}_\infty^2 (\tilde{S}^\top \tilde{S}) - \mathcal{B}_\infty^2 (\mathbb{E}[\tilde{S}^\top \tilde{S}]) \right| 
    + \left| \mathcal{B}_\infty^2 (\mathbb{E}[\tilde{S}^\top \tilde{S}]) - \mathcal{B}_\infty^2 (M_\infty) \right| \\
    &\leq L \left\| \tilde{S}^\top \tilde{S} - \mathbb{E}[\tilde{S}^\top \tilde{S}] \right\|_F 
    + L \left\| \mathbb{E}[\tilde{S}^\top \tilde{S}] - M_\infty \right\|_F
\end{aligned}
\end{equation*}
So that
\begin{equation*}
\begin{aligned}
    \left| \mathbb{E}[\mathcal{B}_\infty^2 (\tilde{S}^\top \tilde{S})] - \mathcal{B}_\infty^2 (M) \right| 
    &\leq L\, \mathbb{E} \left[ \left\| \tilde{S}^\top \tilde{S} - \mathbb{E}[\tilde{S}^\top \tilde{S}] \right\|_F \right] 
    + L \left\| \mathbb{E}[\tilde{S}^\top \tilde{S}] - M \right\|_F \\
    &\to 0.
\end{aligned}
\end{equation*}

Let us define $M := \tilde{S}^\top \tilde{S} \in \mathbb{R}^{T\times T}$ and let's prove the Lipschitzness of the map
$ M \mapsto \mathcal{B}_\infty^2(M)$. We have

\begin{equation*}
    \mathcal{B}^2_\infty(M) = \frac{1}{1-\alpha(M)} (t_1 - \frac{2}{1+\delta(M)}t_2(M) + \frac{1}{(1+\delta(M))^2}t_3(M))
\end{equation*}
where
\begin{align*}
    t_1 &:= \operatorname{Tr}\!\bigl(\tilde\Theta^{\ast^\top}\tilde\Theta^{\ast}\bigr), \quad
    t_2(M) := \operatorname{Tr}\!\bigl(\tilde\Theta^{\ast^\top} A(M) \tilde\Theta^{\ast}\bigr), \\
    t_3(M) &:= \operatorname{Tr}\!\bigl(\tilde\Theta^{\ast^\top} A(M)^2 \tilde\Theta^{\ast}\bigr), \quad
    A(M) := M \left(\tfrac{M}{1+\delta} + \lambda I_T \right)^{-1},\\
    \delta(M)&:=\frac{1}{N}\operatorname{Tr}(M(\frac{M}{1+\delta(M)}+\lambda I_T)^{-1}), \quad \alpha(M) := \frac{\operatorname{Tr}(A(M)^2)}{N(1+\delta(M))^2}.
\end{align*}

\begin{definition}[Lipschitz constant and Uniform bound]
The \emph{Lipschitz constant} of a function $f$ is defined as 
\[
    L_f := \inf \left\{ L > 0 : \lvert f(x) - f(y) \rvert \leq L \lVert x - y \rVert \ \ \forall x,y \right\}.
\]
The \emph{uniform bound} of $f$ is defined as 
\[
    B_f := \sup_{x} \lvert f(x) \rvert .
\]
\end{definition}

Let us first start by showing the Lipschitzness of the map \( M \mapsto \delta(M) \).

\begin{lemma}[Uniform bound and Lipschitzness of \(\delta\)]\label{lem:delta_bound_lip}
Let \(\lambda>0\)  and integers \(N>T\). For \(M\succeq 0 \in \mathbb{R}^{T\times T}\) and \(\delta\ge 0\) define
\[
\varphi(\delta,M):=\frac{1}{N}\operatorname{Tr}\!\left(\Big(\tfrac{M}{1+\delta}+\lambda I_T\Big)^{-1}M\right)-\delta.
\]
For each \(M\succeq0\), let \(\delta(M)\ge0\) be any solution of \(\varphi(\delta(M),M)=0\) (e.g.\ \cite{louart2021concentrationmeasurelargerandom}). Then
\[
0\le \delta(M)\ \le\ \frac{T}{\,N-T\,}\qquad\text{and}\qquad
|\delta(M_1)-\delta(M_2)|\ \le\ \frac{\sqrt{T}}{\lambda\,(N-T)}\,\|M_1-M_2\|_F,
\]
i.e.\ \(B_\delta\le T/(N-T)\) and the map \(M\mapsto \delta(M)\) is Lipschitz (w.r.t.\ \(\|\cdot\|_F\)) with constant \(L_\delta\le \sqrt{T}/(\lambda(N-T))\).
\end{lemma}

\begin{proof}
Fix \(M\succeq 0\) and set
\[
Q:=\Big(\tfrac{M}{1+\delta}+\lambda I_T\Big)^{-1},\qquad
P:=Q-\tfrac{1}{1+\delta}QMQ.
\]
Note that \(Q\succeq 0\) and \(Q^{-1}\succeq \lambda I_T\), hence \(\|Q\|\le \lambda^{-1}\).

\paragraph{Uniform bound.}
At a fixed point \(\varphi(\delta(M),M)=0\), letting \(\{\sigma_i\}_{i=1}^r\) be the nonzero eigenvalues of \(M\) with \(r=\operatorname{rank}(M)\le T\),
\[
\delta
=\frac1N\sum_{i=1}^r \frac{\sigma_i}{\sigma_i/(1+\delta)+\lambda}
=\frac1N\sum_{i=1}^r \frac{(1+\delta)\sigma_i}{\sigma_i+\lambda(1+\delta)}
\le \frac{r}{N}(1+\delta)\le \frac{T}{N}(1+\delta).
\]
Rearranging gives \(\delta \le T/(N-T)\).

\paragraph{Lipschitzness.}
Differentiate \(\varphi\):
\[
\partial_\delta \varphi
=\frac{1}{N(1+\delta)^2}\operatorname{Tr}(QMQM)-1,
\qquad
d_M\varphi[H]=\frac{1}{N}\operatorname{Tr}(P\,H)\quad(\forall H=H^\top).
\]
From the spectral decomposition of \(M\),
\[
\operatorname{Tr}(QMQM)
=\sum_{i:\sigma_i>0}\frac{\sigma_i^2}{\big(\sigma_i/(1+\delta)+\lambda\big)^2}
\le r(1+\delta)^2\le T(1+\delta)^2,
\]
so at a fixed point,
\[
-\partial_\delta\varphi(\delta(M),M)\ \ge\ 1-\frac{T}{N}\ =\ \frac{N-T}{N}\ >0.
\]
Using \(Q^{-1}=\tfrac{M}{1+\delta}+\lambda I_T\),
\[
\frac{1}{1+\delta}QMQ=Q(Q^{-1}-\lambda I_T)Q=Q-\lambda Q^2,
\]
hence \(P=\lambda Q^2\). Therefore
\[
\|P\|\le \lambda\|Q\|^2\le \lambda\cdot \lambda^{-2}=\lambda^{-1},
\qquad
\|P\|_F\le \sqrt{T}\,\|P\|\le \frac{\sqrt{T}}{\lambda},
\]
and thus
\[
|d_M\varphi[H]|
=\frac{1}{N}\,|\operatorname{Tr}(P\,H)|
\le \frac{1}{N}\|P\|_F\|H\|_F
\le \frac{\sqrt{T}}{\lambda N}\,\|H\|_F.
\]
By the implicit function theorem,
\[
D\delta(M)[H]
=-\frac{d_M\varphi(\delta(M),M)[H]}{\partial_\delta\varphi(\delta(M),M)},
\]
whence
\[
|D\delta(M)[H]|
\le \frac{\frac{\sqrt{T}}{\lambda N}}{\frac{N-T}{N}}\,\|H\|_F
= \frac{\sqrt{T}}{\lambda\,(N-T)}\,\|H\|_F.
\]
Taking a supremum over \(\|H\|_F=1\) yields \(L_\delta\le \sqrt{T}/(\lambda(N-T))\).
\end{proof}

\begin{lemma}[Uniform bound and Lipschitzness of $A$]\label{lem:A-global}
Let
\[
A(M):=M\Big(\lambda I_T+\tfrac{M}{1+\delta(M)}\Big)^{-1}, \qquad M\succeq0,
\]
where $\delta(M)$ is defined as in Lemma~\ref{lem:delta_bound_lip}. Then
\[
 B_A \;\le\; 1+B_\delta, 
 \qquad\text{and}\qquad
 L_A \;\le\; \frac{1}{\lambda}+L_\delta.
\]
\end{lemma}
\begin{proof}
Throughout, $\preceq$ the Loewner order.
For brevity write, for $s\in(0,1]$,
\[
Q_s(M):=(\lambda I_T+sM)^{-1},\qquad Q(M):=Q_{s(M)}(M).
\]
Note that $M\succeq0$ implies $0\prec \lambda I_T\preceq \lambda I_T+sM$, hence
$\|Q_s(M)\|\le \lambda^{-1}$ and $0\preceq Q_s(M)\preceq \lambda^{-1}I_T$.

Using $X(\lambda I_T+sX)^{-1}=\frac{1}{s}\big(I_T-\lambda(\lambda I_T+sX)^{-1}\big)$ for any $X\succeq0$, we have
\begin{equation*}\label{eq:A-rewrite}
A(M)=M Q(M)=\frac{1}{s(M)}\big(I_T-\lambda Q(M)\big).
\end{equation*}

\paragraph{Uniform bound.}
From $0\preceq Q(M)\preceq \lambda^{-1}I_T$ we get $0\preceq I_T-\lambda Q(M)\preceq I_T$.
We have
\[
0\preceq A(M)\preceq \frac{1}{s(M)}I_T=(1+\delta(M))\,I_T,
\]
hence $\|A(M)\|\le 1+\delta(M)$ and therefore
\[
B_A\le 1+B_\delta.
\]

\paragraph{Lipschitzness.}
Let $M_1,M_2\succeq0$, and set $s_i:=s(M_i)$ and $\delta_i:=\delta(M_i)$.
Decompose
\[
\begin{aligned}
A(M_1)-A(M_2)
&=\underbrace{M_1(\lambda I_T+s_1 M_1)^{-1}-M_2(\lambda I_T+s_1 M_2)^{-1}}_{\text{(I)}} \\
&\quad+\underbrace{M_2(\lambda I_T+s_1 M_2)^{-1}-M_2(\lambda I_T+s_2 M_2)^{-1}}_{\text{(II)}}.
\end{aligned}
\]

\emph{(I) Fixed $s$ part is $1/\lambda$-Lipschitz.}
For any $X,Y\succeq0$ and fixed $s\in(0,1]$,
\[
X(\lambda I_T+sX)^{-1}-Y(\lambda I_T+sY)^{-1}
=\frac{\lambda}{s}\Big[(\lambda I_T+sY)^{-1}-(\lambda I_T+sX)^{-1}\Big].
\]
By the resolvent identity,
\[
(\lambda I_T+sY)^{-1}-(\lambda I_T+sX)^{-1}
=(\lambda I_T+sY)^{-1}\,s(Y-X)\,(\lambda I_T+sX)^{-1}.
\]
Taking norms and using $\|(\lambda I_T+sZ)^{-1}\|\le \lambda^{-1}$ for $Z\succeq0$ gives
\[
\|X(\lambda I_T+sX)^{-1}-Y(\lambda I_T+sY)^{-1}\|
\le \frac{\lambda}{s}\cdot \frac{1}{\lambda}\cdot s \cdot \frac{1}{\lambda}\,\|X-Y\|
=\frac{1}{\lambda}\|X-Y\|.
\]
Thus $\|(I)\|\le \tfrac{1}{\lambda}\|M_1-M_2\|$.

\emph{(II) Varying $s$.}
Using again the resolvent identity in the $s$-parameter,
\[
Q_{s_1}(M)-Q_{s_2}(M)=(s_2-s_1)\,Q_{s_1}(M)\,M\,Q_{s_2}(M),
\]
and multiplying by $M$ on the left,
\[
M Q_{s_1}(M)-M Q_{s_2}(M)=(s_2-s_1)\,M Q_{s_1}(M)\,M Q_{s_2}(M).
\]
We have $\|M Q_s(M)\|\le \frac{1}{s}$, therefore,
\[
\|M Q_{s_1}(M)-M Q_{s_2}(M)\|
\le |s_2-s_1|\,\|M Q_{s_1}(M)\|\,\|M Q_{s_2}(M)\|
\le \frac{|s_2-s_1|}{s_1 s_2}.
\]
With $s_i=(1+\delta_i)^{-1}$ we have the exact identity
\[
\frac{|s_2-s_1|}{s_1 s_2}
=\left|\frac{1}{s_1}-\frac{1}{s_2}\right|
=|(1+\delta_1)-(1+\delta_2)|
=|\delta(M_1)-\delta(M_2)|.
\]
Hence
\[
\|(II)\|\le |\delta(M_1)-\delta(M_2)|\le L_\delta\,\|M_1-M_2\|.
\]

Combining (I) and (II) yields
\[
\|A(M_1)-A(M_2)\|
\le \Big(\frac{1}{\lambda}+L_\delta\Big)\,\|M_1-M_2\|,
\]
so $L_A\le \frac{1}{\lambda}+L_\delta$, as claimed.
\end{proof}

\begin{lemma}[Bounds and Lipschitzness of $t_2$ and $t_3$]\label{lem:t2t3}
We have
\[
B_{t_2} \le B_A\,t_1, \qquad B_{t_3} \le B_A^2\,t_1.
\]
Moreover, $t_2$ and $t_3$ are Lipschitz in Frobenius norm with
\[
L_{t_2} \le t_1 L_A \qquad
L_{t_3} \le 2 t_1 B_A L_A
\]
\end{lemma}

\begin{proof}
Set $G := \tilde\Theta^{\ast\top}\tilde\Theta^\ast\succeq 0$, so that $\operatorname{Tr}(G)=t_1$.
\paragraph{Uniform bounds.} 
Since $A(M)\succeq 0$,
\[
t_2(M) = \operatorname{Tr}(GA(M))
= \operatorname{Tr}(G^{1/2}A(M)G^{1/2})
\le \|A(M)\|\,\operatorname{Tr}(G)
\le B_A t_1.
\]
Similarly,
\[
t_3(M) = \operatorname{Tr}(GA(M)^2)
\le \|A(M)^2\|\,\operatorname{Tr}(G)
= \|A(M)\|^2 t_1
\le B_A^2 t_1.
\]
\paragraph{Lipschitzness.}
For $M_1,M_2$,
\[
|t_2(M_1)-t_2(M_2)|
= \bigl|\operatorname{Tr}\bigl(G(A(M_1)-A(M_2))\bigr)\bigr|
\le t_1 \|A(M_1)-A(M_2)\|.
\]
Using the Lipschitz property of $A$ gives
\[
|t_2(M_1)-t_2(M_2)| \le t_1 L_A\,\|M_1-M_2\|_F.
\]
Note
\[
A(M_1)^2-A(M_2)^2
= (A(M_1)-A(M_2))A(M_1) + A(M_2)(A(M_1)-A(M_2)).
\]
Hence
\[
\|A(M_1)^2-A(M_2)^2\|
\le (\|A(M_1)\|+\|A(M_2)\|)\,\|A(M_1)-A(M_2)\|
\le 2B_A \|A(M_1)-A(M_2)\|.
\]
Therefore
\[
|t_3(M_1)-t_3(M_2)|
= \bigl|\operatorname{Tr}\bigl(G(A(M_1)^2-A(M_2)^2)\bigr)\bigr|
\le t_1 \|A(M_1)^2-A(M_2)^2\|
\le 2t_1 B_A L_A \|M_1-M_2\|_F.
\]
\end{proof}

\begin{lemma}[Product of bounded Lipschitz functions]\label{lem:prod}
Let \(f,g:\mathcal D\to \mathbb{R}\) satisfy
\[
|f(M)|\le B_f,\qquad
|g(M)|\le B_g
\quad(\forall M\in\mathcal D),
\]
Then the product \(h(M):=f(M)g(M)\) is Lipschitz with constant
\[
\,L_h \leq B_f\,L_g+B_g\,L_f\,.
\]
\end{lemma}

\begin{lemma}[Uniform bound and Lipschitzness of $M\mapsto (1-\alpha(M))^{-1}$]
Let $N>T$, $\lambda>0$, and $M\succeq0\in\mathbb{R}^{T\times T}$.  
Define
$
\alpha(M):=\frac{\operatorname{Tr}(A(M)^2)}{N(1+\delta(M))^2}.$
Then the map
$
M \;\longmapsto\; \frac{1}{1-\alpha(M)}
$
is uniformly bounded and Lipschitz (with respect to $\|\cdot\|_F$), with
\[
B_{(1-\alpha)^{-1}} \le \frac{N}{N-T}, 
\quad 
L_{(1-\alpha)^{-1}} \le \Big(\tfrac{N}{N-T}\Big)^2 \frac{1}{N}\Big[2\sqrt{T}(1+B_\delta)L_A 
+2T(1+B_\delta)^2L_\delta\Big].
\]
\end{lemma}
\begin{proof}
\noindent\textbf{Uniform bound.}  
From Lemma~\ref{lem:A-global}, $\|A(M)\|\le 1+\delta(M)$. Hence
\[
\alpha(M)
=\frac{\operatorname{Tr}(A(M)^2)}{N(1+\delta(M))^2}
\le \frac{T\|A(M)\|^2}{N(1+\delta(M))^2}
\le \frac{T}{N}.
\]
Since $N>T$, this gives $0\le \alpha(M)\le T/N<1$, and therefore
\[
\frac{1}{1-\alpha(M)} \;\le\; \frac{1}{1-T/N} \;=\; \frac{N}{N-T}.
\]

\noindent\textbf{Lipschitzness.} Write $\alpha(M)=\tfrac{1}{N}g(M)h(M)$, where
\[
g(M):=\operatorname{Tr}(A(M)^2), 
\qquad h(M):=(1+\delta(M))^{-2}.
\]
We have 
\[
L_g \le 2\sqrt{T}(1+B_\delta)L_A, 
\quad |g(M)|\le T(1+B_\delta)^2,
\qquad L_h \le 2L_\delta, 
\quad |h(M)|\le 1.
\]
By Lemma~\ref{lem:prod},
\[
L_\alpha \le \tfrac{1}{N}\big(B_g L_h + B_h L_g\big)
= \tfrac{1}{N}\Big[2\sqrt{T}(1+B_\delta)L_A 
+2T(1+B_\delta)^2L_\delta\Big].
\]

Now, the function $x\mapsto (1-x)^{-1}$ has derivative $(1-x)^{-2}$, so over $[0,T/N]$ the Lipschitz factor is at most $(N/(N-T))^2$. Hence
\[
L_{(1-\alpha)^{-1}} \;\le\; \Big(\tfrac{N}{N-T}\Big)^2 L_\alpha.
\]

This proves the claim.
\end{proof}
Now, using \Cref{lem:prod} and the fact that each function is Lipschitz and bounded, we obtain that 
$M \;\mapsto\; \mathcal{B}_\infty^2(M)$
is Lipschitz, with constant $L$ of order \(O(1)\).
That is, as we said above (by defining $M_\infty:= \Sigma_u^{\frac{1}{2}}\operatorname{diag}(\varphi^{(i-T)})\Sigma_u^{\frac{1}{2}}$), we have
\begin{equation*}
\begin{aligned}
    \left| \mathcal{B}_\infty^2 (\tilde{S}^\top \tilde{S}) - \mathcal{B}_\infty^2 (M_\infty) \right| 
    &\leq \left| \mathcal{B}_\infty^2 (\tilde{S}^\top \tilde{S}) - \mathcal{B}_\infty^2 (\mathbb{E}[\tilde{S}^\top \tilde{S}]) \right| 
    + \left| \mathcal{B}_\infty^2 (\mathbb{E}[\tilde{S}^\top \tilde{S}]) - \mathcal{B}_\infty^2 (M_\infty) \right| \\
    &\leq L \left\| \tilde{S}^\top \tilde{S} - \mathbb{E}[\tilde{S}^\top \tilde{S}] \right\|_F 
    + L \left\| \mathbb{E}[\tilde{S}^\top \tilde{S}] - M_\infty \right\|_F
\end{aligned}
\end{equation*}
Thus 
\begin{equation*}
\begin{aligned}
    \left| \mathbb{E}[\mathcal{B}_\infty^2 (\tilde{S}^\top \tilde{S})] - \mathcal{B}_\infty^2 (M) \right| 
    &\leq L\, \mathbb{E} \left[ \left\| \tilde{S}^\top \tilde{S} - \mathbb{E}[\tilde{S}^\top \tilde{S}] \right\|_F \right] 
    + L \left\| \mathbb{E}[\tilde{S}^\top \tilde{S}] - M \right\|_F \\
    &\to 0.
\end{aligned}
\end{equation*}
By plugging in $M_\infty$ and simplifying, we obtain the desired result.
\end{document}